\documentclass[11pt,a4paper,twoside]{article}
\usepackage[UKenglish]{babel}
\usepackage[T1]{fontenc}
\usepackage[utf8x]{inputenc}
\usepackage{latexsym,amsfonts,amsmath,amsthm,amssymb, mathrsfs}
\usepackage{fullpage}
\usepackage{xcolor,bm, bbm}
\usepackage{paralist}
\usepackage{todonotes}
\usepackage{fourier}
\usepackage{amssymb}

\usepackage{tikz}
\usepackage{pgfplots}
\usepackage[labelformat=empty]{caption}
\usetikzlibrary{patterns}

\DeclareMathOperator{\I}{ \mathbb{I} }
\DeclareMathOperator{\U}{  \mbox{Unif}}
\DeclareMathOperator{\vol}{vol}
\DeclareMathOperator{\UOrl}{ X^{(n,V)}} %%% Check if makro is good...
\DeclareMathOperator{\UOrlR}{ Unif \left( \mathbb{B}_R^{(n,V) } \right)}
\DeclareMathOperator{\R}{\mathbb R}
\DeclareMathOperator{\N}{\mathbb N}

\DeclareMathOperator{\Cov }{ Cov }

\DeclareMathOperator{\supp }{ supp }

\newtheorem{thm}{Theorem}[section]
\newtheorem{cor}[thm]{Corollary}
\newtheorem{df}[thm]{Definition}
\newtheorem{proposition}[thm]{Proposition}
\newtheorem{rem}[thm]{Remark}
\newtheorem{lem}[thm]{Lemma}
\newtheorem{thmalpha}{Theorem}

{
\theoremstyle{definition}

}

\newtheorem{thmbeta}{Theorem}

{
\theoremstyle{definition}
}
\newtheorem{Ass}[thmbeta]{Assumption}

\allowdisplaybreaks
\begin{document}
\title{Sharp concentration phenomena in high-dimensional Orlicz balls}
\author{Lorenz Frühwirth and Joscha Prochno}
\date{}

\maketitle

\begin{abstract}
In this article, we present a precise deviation formula for the intersection of two Orlicz balls generated by Orlicz functions $V$ and $W$. Additionally, we establish a (quantitative) central limit theorem in the critical case and a strong law of large numbers for the "$W$-norm" of the uniform distribution on $\mathbb{B}^{(n,V)}$. Our techniques also enable us to derive a precise formula for the thin-shell concentration of uniformly distributed random vectors in high-dimensional Orlicz balls. In our approach we establish an Edgeworth-expansion using methods from harmonic analysis together with an exponential change of measure argument.
\end{abstract}

% % % % % % % % % % % % % % % 
\section{Introduction \& Main results}
% % % % % % % % % % % % % % %
Let $n \in \N,1 \leq p < \infty$ and let $\mathbb{D}^{(n,p)}$ be the volume normalized $\ell_p$-ball in $\R^n$. In \cite{SS1991} Schechtman and Schmuckenschläger studied the asymptotic behavior of the intersection of $\mathbb{D}^{(n,p)} $ with an appropriately scaled $\ell_q$-ball and found the threshold behavior
\begin{equation}
\label{Eq_vol_int}
\vol_{n} \left( \mathbb{D}^{(n,p)} \cap t \mathbb{D}^{(n,q)} \right) \xrightarrow{n \rightarrow \infty} 
\begin{cases}
    0, & t < A_{p,q}^{-1}\\
    1, & t > A_{p,q}^{-1},
\end{cases}
\end{equation}
where $A_{p,q} \in (0, \infty)$ is an explicit constant. Their proof mainly relied on a probabilistic representation of the uniform distribution on $\ell_p^n$-balls obtained by Schechtman and Zinn \cite{SZ1990} and independently by Rachev and Rüschendorf \cite{RSR1991}). More precisely, for $X^{(n,p)} \sim \mathbb{D}^{(n,p)}$, we can write
\begin{equation}
\label{Eq_Prob_Rep}
    X^{(n,p)} \stackrel{d}{=} U^{1/n} \frac{\left( Y_1, \ldots, Y_n \right)}{|| \left( Y_1, \ldots, Y_n \right) ||_p},
\end{equation}
where $U \sim \U ([0,1])$ and $(Y_i)_{i \in \N}$ are independent and identically distributed (iid) with Lebesgue density
\begin{equation*}
   \R \ni x \mapsto \frac{1}{2 p^{1/p} \Gamma(1+ 1/p )}e^{ -|x|^p/p}.
\end{equation*}
It is then possible to represent the volume in \eqref{Eq_vol_int} as the probability that a random vector, which is uniformly distributed in the (normalized) $\ell_p^n$-ball, has $\ell_q$-norm smaller than $t$. The probabilistic representation of $X^{(n,p)}$ allows us to work with independent random variables and thus the limit in \eqref{Eq_vol_int} is essentially an application of the strong law of large numbers. The threshold case $t=A_{p,q}$ first remained an open problem before Schmuckenschläger proved a central limit theorem for the $\ell_q$-norm of a point chosen randomly in an $\ell_p^n$-ball \cite{schmuckenschlager2001clt}. The existence of a probabilistic representation as in \eqref{Eq_Prob_Rep} gave rise to the rich and flourishing theory of high-dimensional $\ell_p$-geometry, which was intensively studied over the past 15 years (we refer the interested reader to the surveys \cite{Prochnoldpmdp,PTTSurvey}). One aim of the present paper is to generalize the Schechtman--Schmuckenschläger result from \eqref{Eq_vol_int} in a different direction. All mentioned results above heavily use the representation in \eqref{Eq_Prob_Rep}, which is limited to $\ell_p$-balls. Here we want to study concentration phenomena like \eqref{Eq_vol_int} for the intersection of so-called Orlicz balls. Typically, taking a symmetric and convex function $V: \R \rightarrow [0, \infty)$ with $V(x)=0$ if and only if $x=0$ and some $R \in (0, \infty)$, the associated Orlicz ball with "radius" $R$ is defined as 
\begin{equation}
    \label{Def_Orlicz_ball}
    \mathbb{B}^{(n,V)}_R := \left\{ x \in \R^n \,:\, \sum_{i=1}^n V(x_i) \leq n R \right\};
\end{equation}
in this work we deal with a slightly more restrictive notion of an Orlicz function (see Definition \ref{Definition_Orlicz_ball} and Remark \ref{Rem_Def_Orl}). Orlicz spaces serve as natural generalizations of $\ell_p$-spaces and are fundamental elements within the class of symmetric Banach sequence spaces. They have been intensively studied
in the functional analysis literature and we direct the interested reader to \cite{HAO2006303, Kam1984,Kosmol2011Opt,Kwapien1985, Schuett1994,PS2012Comb, RS1988} and the references cited therein. In \cite{KP2021}, using the maximum entropy principle from statistical mechanics and elements from (sharp) large deviations theory, Kabluchko and Prochno were able to bypass the non-existence of a classical Schechtman--Schmuckenschläger-type probabilistic representation and computed the asymptotic volume of high-dimensional Orlicz balls; independently similar problems were considered by Barthe and Wolff in \cite{barthe2023volume} with a view towards confirming the Kannan--Lov\'asz--Simonovits (KLS) spectral gap conjecture. Although not stated explicitly it follows from the proof of \cite[Theorem B]{KP2021} that the convergence rate is exponential. In our Theorem \ref{ThmSLD} we are able to present a precise version of \cite[Theorem B]{KP2021}, which also includes lower-order terms and could not be obtained by the approach developed in \cite{KP2021}. The study of moderate and large deviations principles in the realm of asymptotic geometric analysis has been relatively recent, pioneered by Gantert, Kim, and Ramanan in \cite{gantert2017}, and further developed by Prochno, Kabluchko, and Thäle in \cite{KPT2019_II}. In their recent work \cite{alonso2021}, Alonso-Gutiérrez, Prochno, and Thäle uncovered a connection between the investigation of moderate and large deviations for isotropic log-concave random vectors and the famous KLS conjecture. Unlike central limit theorems, (sharp) large deviation principles are highly sensitive to the distribution of the underlying random variables. 
Although large deviations principles (LDPs), unlike concentration results or large deviation upper bounds, precisely identify the asymptotic exponential decay rate and allow for the identification of conditional limit laws (see \cite{CondLim_FP, RKConLim}), they have a notable drawback. In general, LDPs only provide approximate estimates of the probability that a sequence of random variables lies in an atypical region. Specifically, they characterize the limit of the logarithms of the deviation probabilities as the dimension $n$ tends to infinity. Sharp LDPs, on the other hand, not only provide the precise exponential decay but also include additional subexponential factors, offering more accurate estimates and are more in the spirit of asymptotic geometric analysis.
\par{}
While numerous articles investigate the asymptotic behavior of high-dimensional $\ell_p$-balls and spheres (see, e.g., \cite{apt2018, gantert2016, gantert2017, kabluchko2021, KPT2019_I, KPT2019_II, kaufmann2021,RKConLim}), the study becomes more nuanced and delicate when extending to high-dimensional Orlicz balls. This can be seen in the work of Kabluchko and Prochno \cite{KP2021} and the complexity stems from the absence of a probabilistic representation for the uniform distribution on $\mathbb{B}^{(n,V)}_R$ and because general Orlicz functions are not homogenous. The former problem can be resolved by employing suitable Gibbs measures $\mu_{\alpha}$, distributed according to the Lebesgue density
\begin{equation*}
\frac{d \mu_{\alpha}}{dx}(x) = \frac{e^{\alpha V(x)}}{Z_{\alpha}}, \quad x \in \mathbb{R},
\end{equation*}
where, in statistical mechanics terminology, $\alpha \in (-\infty, 0)$ is referred to as the inverse critical temperature, and $Z_{\alpha}$ is the partition function, i.e., the normalizing constant; this approach was initiated in \cite{KP2021} and has led to successful breakthroughs in addressing geometric questions within high-dimensional Orlicz spaces by probabilistic means (see, e.g., \cite{AP2022,CondLim_FP,PJ2020OrliczLim,kim2022,RKConLim}).  
Turning to sharp large deviations, those can be traced back to Bahadur and Rao \cite{BahadurRao1960} and Petrov \cite{PetrovSDev} (see also \cite{eichelsbacher2003}). Essentially, they used an exponential change of measure to identify the exponential decay of the probability that the empirical mean of iid random variables exceeds their common mean. The residual term can then be approximated using a first-order Edgeworth expansion (see, e.g., \cite[p. 293]{PetrovSDev}). The proof of Theorem \ref{ThmSLD} adopts a similar approach, but encounters certain challenges in establishing an Edgeworth expansion for the resulting sum of random variables (see Section \ref{Section:Edgeworth} and the discussion therein). After identifying the precise asymptotic formula for the Schechtman--Schmuckenschläger ratio \eqref{Eq_vol_int} in the case of Orlicz balls, one may raise questions about its behavior in the critical case; note that this question could not be addressed by the methods in \cite{KP2021}. This inquiry is addressed in our Theorem \ref{Thm_SLLN_CLT}, where we demonstrate that the normalized volume intersection for high-dimensional Orlicz balls adheres to a quantitative Central Limit Theorem with a speed of order $n^{-1/2}$, analogous to the classical Berry--Esseen theorem.
\par{}
Our techniques also allow to give a precise description of the so-called thin-shell concentration of high-dimensional Orlicz balls (see Theorem \ref{ThmSLD_thinshell}). Roughly speaking, for small values of $\delta \in (0, \infty)$, we provide a precise formula for 
\begin{equation}
\label{Eq_thinshell}
    \mathbb{P} \left[ \left | \frac{\| X^{(n,V)} \|_2^2}{n} -  m \right | \geq \delta  \right],
\end{equation}
where $X^{(n,V)}$ follows a uniform distribution on an Orlicz ball generated by the Orlicz function $V$. Estimating probabilities of the form \eqref{Eq_thinshell} has attracted considerable attention since the works of Sudakov \cite{sudakov78}, Diaconis and Freedman \cite{diaconis1984}, von Weizsäcker \cite{vonWeiz1997}, and Antilla, Ball, and Perissinaki \cite{anttila2003}. We also highlight the breakthrough result by Klartag \cite{klartag2007}, who proved that the marginals of an isotropic random vector are close to a Gaussian distribution. In \cite{anttila2003}, it was established that any isotropic (with mean $0$ and identity covariance matrix) random vector $X^{(n)} \in \R^n$ has (approximately) Gaussian marginals if $ \|X^{(n)} \|_2 / \sqrt{n}$ concentrates around $1$. In other words, $||X^{(n)}||_2$ concentrates in a thin-shell with radius $\sqrt{n}$. This principle has led to the thin-shell width conjecture, which proposes that there exists an absolute constant $C \in (0, \infty)$ such that for every $n \in  \N$ and every
isotropic random vector $X^{(n)} \in \R^n$ one has $ \mathbb{E} \left[ \| X^{(n)} \|_2 - \sqrt{n}\right]^2 \leq C$. This is equivalent to the famous variance conjecture, first formally proposed in \cite{bobkov2003} (see also \cite{alonso2013}). It asserts the existence of an absolute constant $C \in (0, \infty)$ such that for any isotropic random vector $X^{(n)} \in \R^n$, it holds that $\mathbb{V}\left[ \|X^{(n)} \|_2^2 \right] \leq C n$. The variance conjecture has been successfully approached for various important examples of isotropic convex bodies (see, e.g., \cite{wojtaszczyk2007}, \cite{kolesnikov2018} and the references cited therein). The variance conjecture is a special case of the renowned KLS conjecture, which roughly says that, up to a universal constant $C \in (0, \infty)$ independent of the dimension $n$, the best way to divide a convex body into two pieces of equal mass in order to minimize their interface is a hyperplane bisection. For a long time the best known bound was obtained by Bourgain \cite{Bourgain91, Bourgain2003} where it is proven that $C = O \left( n^{1/4} \log n \right)$. A few years ago Chen \cite{Chen21} accomplished a breakthrough by obtaining a subpolynomial bound for the KLS constant. This work led to great further improvements (e.g., by Klartag and Lehec \cite{KlarLeh2022}) with the currently best known estimate for the KLS constant of $O( \sqrt{\log n})$ obtained by Klartag in \cite{klartag2023} (see also the recent survey \cite{klarLehec2024} for a more comprehensive discussion on the KLS conjecture and related topics).
\par{}
Returning to thin-shell concentration estimates, we mention that in 2018, Lee and Vempala \cite{lee2018} proved that for $t  \in (0,\infty)$,
\begin{equation*}
\mathbb{P} \left[ \left| \frac{\|X^{(n)}\|_2}{\mathbb{E}\left[ \|X^{(n)}\|_2\right]} -1 \right | \geq t \right] \leq e^{-c \min \left \{ t, t^2 \right \} \sqrt{n}} \quad \text{and} \quad
    \mathbb{E} \left[ \|X^{(n)} \|_2 - \sqrt{n} \right]^2 \leq C \sqrt{n};
\end{equation*}
another estimate of a similar order was obtained by Guédon and Milman \cite{guedon2011} (see also \cite{naor2007}, \cite{schechtman2000}, and \cite{alonso2018} for more precise estimates in case of $\ell_p$-balls). More recently, Prochno and Alonso-Gutierrez \cite{AP2022}, using a moderate deviations perspective, were able to provide precise estimates for probabilities of the form
\begin{equation*}
    \mathbb{P} \left[ \left| \frac{\| X^{(n)} \|_2^2}{n L_Z^2} - 1 \right | \geq t_n \right],
\end{equation*}
where $X^{(n)}$ is uniformly distributed in a high-dimensional Orlicz ball $\mathbb{B}^{(n,V)}_R$, $L_Z^2$ is the asymptotic value of the isotropic constant (see also \cite[Theorem A]{AP2022}) and $(t_n)_{n \in \N}$ is a sequence that grows to infinity slower than $\sqrt{n}$. Their approach requires the Orlicz function $V$ to grow sufficiently fast; more precisely, they assumed that
\begin{equation*}
\lim_{x \rightarrow \infty} \frac{V(x)}{x^2}= \infty.
\end{equation*}
In our Theorems \ref{ThmSLD} and \ref{ThmSLD_thinshell}, we are able to weaken this assumption and we only require $V$ to grow at least as fast as $x \mapsto x^2$ (see also Assumption \ref{Ass_A}).

% % % % % % % % % % % % % % %
\subsection{Main results}
% % % % % % % % % % % % % % %

Before stating our main results, Theorem \ref{ThmSLD}, \ref{Thm_SLLN_CLT} and \ref{ThmSLD_thinshell}, we introduce a few quantities. An Orlicz function is a symmetric and convex function $V: \R \rightarrow [0,\infty)$ with $V(x)=0$ if and only if $x=0$, which is essentially $\mathcal{C}_2$ (see Section \ref{sec:Orlicz_Gibbs} for a precise definition). The associated Orlicz ball $\mathbb{B}_R^{(n,V)}$ with radius $R \in (0,\infty)$ in $\R^n$ is then defined as in \eqref{Def_Orlicz_ball}. For $n \in \N$ and $R \in (0, \infty)$, let $ X^{(n,V)} \sim \UOrlR$. Then, as shown in Theorem \ref{Thm_SLLN_CLT}, for any (linearly independent) Orlicz functions $V$ and $W$, the almost sure limit
\begin{equation}
\label{Eq_W_Expectation}
m := \lim_{n \rightarrow \infty} \frac{1}{n} \sum_{i=1}^n W \left( X_i^{(n,V)} \right) \in (0, \infty)
\end{equation}
exists.
\par{}
Let $V,W: \R \rightarrow [0,\infty)$ be two Orlicz functions and, for $(\alpha, \beta)  \in  \R^2$, we define a measure $\mu_{\alpha, \beta}$ via the Lebesgue density
\begin{equation*}
    \frac{d \mu_{\alpha, \beta}(x)}{d  x} = \frac{ e^{\alpha V(x) + \beta W(x)}}{Z_{\alpha, \beta}}, \quad x \in \R,
\end{equation*}
provided $Z_{\alpha, \beta} := \int_{\R} e^{\alpha V(x) + \beta W(x)} dx < \infty$. Gibbs measures of this or a similar form have proven to be very useful when dealing with the asymptotic behavior of high-dimensional Orlicz balls (for more background on the maximum entropy principle for Orlicz spaces, we refer to Section 1.2 in \cite{KP2021}). Moreover, we define $\varphi: \R^2 \rightarrow (-\infty,\infty]$ with
\[\varphi( \alpha, \beta) := \log Z_{\alpha, \beta }.\]
Let $\I : \R^2 \rightarrow [0, \infty]$ be the Legendre transform of $\varphi$, i.e., for $ (x,y) \in \R^2$, we set
\begin{equation}
\label{Eq_GRF}
    \I(x,y):= \sup_{(\alpha, \beta) \in \R^2} \left[ \alpha x + \beta y - \varphi( \alpha, \beta) \right],
\end{equation}
which is well defined since $\varphi$ is convex.

\begin{Ass}
\label{Ass_A}
Let $V,W : \R \rightarrow [0, \infty)$ be two Orlicz functions. We say $V$ and $W$ satisfy Assumption \ref{Ass_A}, if 
\begin{equation*}
  \liminf_{x \rightarrow \infty} \frac{V(x)}{W(x)} > 0
\end{equation*}
\end{Ass}
Assumption \ref{Ass_A} essentially means that $W$ grows at most as fast as $V$.
\begin{Ass}
 \label{Ass_B}
 Let $V,W : \R \rightarrow [0, \infty)$ be two Orlicz functions. We say $V$ and $W$ satisfy Assumption \ref{Ass_B}, if 
 \begin{equation*}
     \lim_{x \rightarrow \infty} \frac{V(x)}{W(x)} = \infty .
 \end{equation*}
 \end{Ass}
Assumption \ref{Ass_B} requires $V$ to grow much faster than $W$ and we note that Assumption \ref{Ass_B} implies \ref{Ass_A}. We highlight that similar conditions to Assumption \ref{Ass_B} were already imposed in, e.g., \cite{AP2022,RKConLim}.
\begin{rem}
For example, Assumption \ref{Ass_B} is satisfied for $V(x) = |x|^p$ and $W(x)= |x|^q$ if $p > q \geq 1$. The Orlicz functions $V(x):=x^4+x^2$ and $W(x):=x^4$ satisfy Assumption \ref{Ass_A}, but not Assumption \ref{Ass_B} (see also Remark \ref{Rem_Cond_thm}, where we demonstrate the necessity of distinguishing between Assumptions \ref{Ass_A} and \ref{Ass_B}!). 
\end{rem}
In the theory of large deviations, functions as $\I$ from \eqref{Eq_GRF} are usually referred to as good rate functions (GRF). This means that all level sets of $\I$ are compact.

\begin{thmalpha}
\label{ThmSLD}
$(i)$ Let $R \in (0, \infty)$ and $V,W: \R \rightarrow [0, \infty)$ be linearly independent Orlicz functions that satisfy Assumption \ref{Ass_A}. Then there exists an $\varepsilon \in (0, \infty)$ such that for all $t \in (m, m + \varepsilon)$, we have
\begin{equation}
\label{Eq_SharpDev_V_W}
    \frac{ \vol_{n} \left( \mathbb{B}_R^{(n,V)} \cap \mathbb{B}_t^{(n,W)} \right)}{\vol_{n} \left( \mathbb{B}_R^{(n,V)} \right)} = 1- \frac{1}{\sqrt{2 \pi n} \sigma}   e^{ -n \I(R,t)  } (1+o(1)),
\end{equation}
where $m \in (0, \infty)$ is the constant from \eqref{Eq_W_Expectation}, $\I $ is the GRF from \eqref{Eq_GRF} and $\sigma = \sigma(V,W,R,t)$ is an explicit positive quantity.
\par{}
$(ii)$ If $V,W: \R \rightarrow [0, \infty)$ satisfy Assumption \ref{Ass_B}, then \eqref{Eq_SharpDev_V_W} holds for all $t \in (m, W(V^{-1}(R))$.
\end{thmalpha}
% \begin{rem}
%     In Theorem \ref{ThmSLD}, $V^{-1}$ denotes the inverse of $V$ when restricted on the non-negative real numbers. $V^{-1}$ exists, since $V: [0, \infty) \rightarrow [0, \infty)$ is invertible as a continous and strictly increasing mapping.
% \end{rem}

\begin{rem}
We observe that in the case of $W(x):= |x|^p$ for some $p \geq 1$, we can express
\begin{align*}
    \mathbb{P} \left[ \left | \left| X^{(n,V)} \right | \right |_p^p \geq t \right] &= 1 - \frac{ \vol_{n} \left( \mathbb{B}_R^{(n,V)} \cap \mathbb{B}_t^{(n,W)} \right)}{\vol_{n} \left( \mathbb{B}_R^{(n,V)} \right)}. 
\end{align*}
Hence, Theorem \ref{ThmSLD} generalizes \cite[Theorem 11]{kaufmann2021} and provides a refinement of \cite[Theorem B]{KP2021}. In several related works (see, e.g.,\cite{kaufmann2021,liao2024geometric}) the admissible values of $t$ are non-explicit. In Theorem \ref{ThmSLD} $(ii)$ we can identify the explicit interval $(m, W(V^{-1}(R)))$ such that the sharp asymptotic formula given in \eqref{Eq_SharpDev_V_W} holds for all values of $t$ in that interval. We consider this to be a non-trivial feature of our result (see the proof of Lemma \ref{Lem_AuxL1}).  
\end{rem}
% We also obtain the following corollary, which improves upon results in \cite[Section 3.3.1]{kim2022}.
% \begin{cor}
%     \label{Cor_subquad}
%     Let $V: \R \rightarrow [0,\infty)$ be an Orlicz function which is either subquadratic or superquadratic, i.e., we have 
%     \[
%     \lim_{x \rightarrow \infty} \frac{V(x)}{x^2} = \infty \qquad \text{or} \qquad \lim_{x \rightarrow \infty} \frac{V(x)}{x^2} = 0.
%     \]
%     Then $\left( \frac{||X^{(n,V)}||_2}{\sqrt{n}} \right)_{n \in \N}$ satisfies a large deviation principle with some GRF $J_X: \R \rightarrow [0, \infty)$.
% \end{cor}
% Corollary \ref{Cor_subquad} is of interest since 

As in case of $\ell_p$-balls, we see that the "$W$-norm" of a uniform distribution on a $V$-Orlicz ball $\mathbb{B}^{(n,V)}_R$ concentrates around a critical value $m \in (0, \infty)$. The following theorem shows that this concentration happens almost surely and, under a suitable normalization, we obtain a Berry--Esseen-type result, i.e., a quantitative central limit theorem. We denote the Kolmogoroff distance by $d_{\text{Kol}}$ (see Equation \eqref{Eq_Kol_dist} for a formal definition). 

\begin{thmalpha}(Strong law of large numbers and CLT)
\label{Thm_SLLN_CLT}
Let $V,W: \R \rightarrow [0, \infty)$ be linearly independent Orlicz functions. For $n \in \N$ and $R \in (0, \infty)$, let $\UOrl \sim \UOrlR  $, then the limit
\begin{equation}
\label{Eq_SLLN}
    m= \lim_{n \rightarrow \infty} \frac{1}{n} \sum_{i=1}^n W \left (X_i^{(n,V)} \right)
\end{equation}
exists almost surely with $m = \mathbb{E}_{\mu_{\alpha^{*}, 0}} [W(X)]$, where $\alpha^{*} \in (0, \infty)$ is chosen such that $\mathbb{E}_{\mu_{\alpha^{*}, 0}} [V(X)] = R$. Moreover, we can even give a quantitative version of the limit above: there exists a constant $c \in (0, \infty)$ such that
\begin{equation}
\label{Eq_CLT}
      d_{\text{Kol}} \left(  \frac{1}{ \sqrt{n}} \sum_{i=1}^n \left[ W \left(X_i^{(n,V)} \right) -m \right] , Z \right) \leq \frac{c}{\sqrt{n}}, \qquad n \in \N,
\end{equation}
where $Z \sim \mathcal{N}(0, \sigma_{*}^2)$ and $\sigma_{*}^2 \in (0, \infty)$ is a different quantity as in Theorem \ref{ThmSLD_thinshell}. 
\par{}
Moreover, the result remains valid if we choose $V(x)= |x|^p$ and $W(x)=|x|^q$ for some $p,q \in (0,1)$ with $p \neq q$.
\end{thmalpha}

\begin{rem}
It's worth noting that Theorem \ref{Thm_SLLN_CLT} holds true for functions $V(x)=|x|^p$ and $W(x)=|x|^q$, where $p, q \in (0,1)$. Remarkably, these functions are neither Orlicz functions nor convex. We emphasize that a key component of the proof of Theorem \ref{Thm_SLLN_CLT} is the establishment of an Edgeworth expansion for a specific sequence of random vectors (see Section \ref{Section:Edgeworth}). We anticipate that the techniques developed in Section \ref{Section:Edgeworth} could apply to a broader class of functions beyond those considered in Theorem \ref{Thm_SLLN_CLT}, which could be of independent interest.
\end{rem}

As a consequence of our CLT, we also obtain the following Corollary answering a question in \cite[p. 4]{KP2021}.

\begin{cor}
    \label{Cor_Clt}
    Let $V,W: \R \rightarrow [0,\infty)$ be two linearly independent Orlicz functions. Then, for any radius $R \in (0, \infty)$, we have
    \[
    \frac{\vol_{n}\left( \mathbb{B}_{R}^{(n,V)} \cap \mathbb{B}_{m}^{(n,W)}\right)}{\vol_{n} \left(\mathbb{B}_{R}^{(n,V)} \right)}  \stackrel{n \rightarrow \infty}{\longrightarrow} \frac{1}{2},
    \]
    where $m = m(R) \in (0, \infty)$ is the quantity from \eqref{Eq_W_Expectation}.
\end{cor}

% \todoLF[inline]{Man kann den folgenden Satz für jede andere Orliczfunktionen $W$ formulieren, aber wsl ist die $L_2$-Version am interessantesten. Ich denke unsere Methoden würden es auch hergeben, für ein kleines $\varepsilon > 0$, folgende Aussage zu bekommen:
% \begin{equation*}
%  \mathbb{P} \left[ \left|  \frac{|| X^{(n,V)} ||_2^2}{n}  - m \right | > \varepsilon \right] = \frac{1}{\sqrt{2 \pi n} \sigma_1}e^{-n \mathbb{I}(R,m+ \varepsilon)} (1 + o(1))+   \frac{1}{\sqrt{2 \pi n} \sigma_2}e^{-n \mathbb{I}(R,m- \varepsilon)} (1 + o(1)).  
% \end{equation*} }

Another valuable application of our techniques is the following thin-shell concentration result, which complements \cite[Theorem B]{AP2022}.
\begin{thmalpha}(Sharp thin-shell concentration)
\label{ThmSLD_thinshell}
Let $R \in (0, \infty)$, $W(x):= x^2$ and $V: \R \rightarrow [0, \infty)$ be an Orlicz function such that Assumption \ref{Ass_A} is satisfied. Then, there exists an $\varepsilon \in (0, \infty)$ such that for all $\delta \in (0, \varepsilon)$, we have
\begin{equation}
\label{Eq_SharpDev_thinshell}
    \mathbb{P} \left[ \left | \frac{|| X^{(n,V)} ||_2^2}{n}   - m \right | > \delta \right] = \frac{1}{\sqrt{2 \pi n} \tilde{\sigma}}e^{-n \min \{\mathbb{I}(R,m+ \delta ), \mathbb{I}(R,m- \delta ) \}} (1 + o(1))
\end{equation}
where $\UOrl \sim \UOrlR  $ and $\I$ is the GRF from \eqref{Eq_GRF} and $m \in (0,\infty)$ is the almost sure limit from Theorem \ref{Thm_SLLN_CLT} (see also Equation \ref{Eq_SLLN}). The quantity $\tilde{\sigma} \in (0, \infty) $ depends on $V, R$, and $\delta$.
\end{thmalpha}
\begin{rem}
 In fact, we will prove a more general result as it is stated in Theorem \ref{ThmSLD_thinshell}, where the $2$-norm is replaced by the "$W$-norm" with an Orlicz function $W$ satisfying Assumption \ref{Ass_A}.
 
 In particular, we note that Assumption \ref{Ass_A} is weaker than the superquadratic assumption required in \cite[Theorem B]{AP2022}.
\end{rem}
\section{Notation and Preliminaries}\label{sec:notation and prelim}
% % % % % % % % % % % % % % % % % % % % % % % % % % % % % % % % % % % % 

% % % % % % % % % % %
\subsection{Notation}
% % % % % % % % % % %

Let $X$ be a real-valued random variable and $\mu$ be some distribution. If $X$ is distributed according to $\mu$, we write $X \sim \mu$. We denote the Dirac-measure in some point $x \in \R$ by $\delta_x$. The mean of $X$ is denoted by $\mathbb{E} \left[X \right]$ and for the variance of $X$ we write $\mathbb{V} \left[X \right]$, provided the respective quantities exist. Let $f : \R \rightarrow \R$ be some measurable function and $\mu$ some distribution; whenever we write 
\begin{equation*}
 \mathbb{E}_{\mu}\left[ f(X) \right]   \quad \text{or} \quad \mathbb{V}_{\mu}\left[ f(X) \right], 
\end{equation*}
we indicate that $X \sim \mu$. Let $k \in \N$ and $\Sigma \in \R^{k \times k}$ be a positive definite, symmetric matrix and $v \in \R^k$ be some vector. The (multivariate) normal distribution with mean $v$ and covariance matrix $\Sigma$ is then denoted by $ \mathcal{N} \left( \mu , \Sigma \right)$. Throughout this article we will only see the cases $k=1$ or $k=2$. Let $ X$ and $Y$ be random variables. The \textit{Kolmogoroff distance} between $X$ and $Y$ is defined as
\begin{equation}
\label{Eq_Kol_dist}
    d_{\text{Kol}}(X,Y):= \sup_{t \in \R} \left |  \mathbb{P} \left[ X \leq t\right] - \mathbb{P} \left[ Y \leq t \right] \right | . 
\end{equation}
Let $I \subseteq \R$ be an interval with endpoints $a,b$. In Section \ref{Section:Edgeworth} we will frequently use the notation $|I| = b-a$ to indicate the length of $I$. Let $A \subseteq \R^d$ be a set, where, depending on the context, we will use both notations $\mathbb{1}_A(x)$ and $\mathbb{1}_{\left[ x \in A \right]}$ for the indicator function over the set $A$.
Let $U \subseteq \R^2$ be an open non-empty set and $\varphi: U \rightarrow \R$ be two times continuously differentiable on $U$. We then write $\varphi \in \mathcal{C}^2(U)$ or if $U = \R^2$, we just write $\varphi \in \mathcal{C}^2$. For $(x,y) \in U$, the gradient of $\varphi $ in $(x,y)$ is defined as
\begin{equation*}
    \nabla \varphi (x,y):= \left( \begin{array}{c}
         \frac{\partial \varphi}{\partial x} (x,y) \\
          \frac{\partial \varphi}{\partial y} (x,y)
    \end{array} \right).
\end{equation*}
The Hessian of $\varphi$ in some point $(x,y) \in U$ is denoted by 
\begin{align*}
H_{\varphi}(x,y):=\left( \begin{matrix}
         \frac{\partial^2 \varphi}{\partial x^2} (x,y) &  \frac{\partial^2 \varphi}{\partial x \partial y} (x,y) \\
          \frac{\partial \varphi}{\partial y \partial x } (x,y) & \frac{\partial^2 \varphi}{\partial y^2} (x,y)
    \end{matrix} \right).
\end{align*}

Let $(a_n)_{n \in \N}$ be a sequence of real numbers and $(b_n)_{n \in \N}$ a sequence of monotone decreasing and positive numbers. We write $a_n = o(b_n)$, if, for all $\varepsilon \in (0, \infty)$, there exists an $N \in \N$ such that for all $n \geq N$, it holds that
\begin{equation*}
    |a_n| \leq \varepsilon b_n.
\end{equation*}
We write $a_n = O(b_n)$, if there exists a constant $C \in (0, \infty)$ such that, for all $n \in \N $
\begin{equation*}
    |a_n| \leq C b_n.
\end{equation*}
If all $a_n$ are non-negative, we also use the Vinogradov notation $a_n \ll b_n$ if $a_n = O(b_n)$. Let $(X_n)_{n \in \N}$ be a sequence of real-valued random variables. 

\subsection{Orlicz functions \& Gibbs measures}
% % % % % % % % % % % % % % % % % % % % % % % % % 

\label{sec:Orlicz_Gibbs}

\begin{df}
\label{Definition_Orlicz_ball}
A symmetric, convex function $V: \R \rightarrow [0, \infty)$ with $V(x)=0$ iff $x =0$ is called an Orlicz function if $V$ is two times continously differentiable. In fact, we only need that $V$ is $\mathcal{C}^2$ except of countably many points that do not accumulate anywhere. We recall the definition of the associated Orlicz ball $\mathbb{B}^{(n,V)}_R$ of radius $R \in (0, \infty)$ with
\begin{equation*}
\mathbb{B}_R^{(n,V)} = \left \{  x \in \R^n : \sum_{i=1}^n V(x_i) \leq R n \right \}.
\end{equation*}
\end{df}
\begin{rem}
\label{Rem_Def_Orl}
Our notion of an Orlicz function is slightly more restrictive than usual, since we require that $V$ is $\mathcal{C}^2$ except of a discrete set of points. Often (see, e.g., \cite{AP2022, KP2021}) $V$ is assumed to be just convex, which by Alexandrov's theorem, implies that $V$ is $\mathcal{C}^2$ almost everywhere. The set of points where $V$ is not $\mathcal{C}^2$ could still be dense. Here we exclude such examples. We believe that this assumption could be overcome with a significant amount of technical work in Section \ref{Section:Edgeworth}, which we have chosen not to undertake.
\end{rem}

The uniform distribution of the intersection of two Orlicz balls $ \mathbb{B}_R^{(n,V)}$ and  $ \mathbb{B}_R^{(n,W)}$ is naturally related to two-factor Gibbs measures $\mu_{\alpha, \beta}$ defined via the Lebesgue density
\begin{equation}
\label{Eq_Def_GibbsMeas}
\frac{d \mu_{\alpha, \beta}}{dx}(x) = \frac{e^{\alpha V(x)+ \beta W(x)}}{Z_{\alpha, \beta}}, \qquad  x \in \R,    
\end{equation}
where 
\[
Z_{\alpha, \beta} = \int_{\R} e^{\alpha V(x) + \beta W(x)}dx
\]
is the normalizing constant and $(\alpha, \beta) \in \R^2$ are chosen in a way such that $\mu_{\alpha, \beta}$ exists. For instance, if we choose $\alpha=-1$ and $\beta=0$, we always have that $\mu_{\alpha, \beta}$ exists since $V$ is convex as Orlicz function. 
% Let us recall here what an Orlicz function is and introduce Gibbs measures whose potentials are given by Orlicz functions.

% \begin{df}
% 		An Orlicz function $V: \R \rightarrow [0,\infty)$ is a convex, symmetric function with $V(0) = 0$ and $V(x) > 0$ for $ x \neq 0$. We define the associated Orlicz ball 
% 		\begin{equation*}
% 		\mathbb{B}_R^{n,V} := \Big \{ x \in \R^n \ : \ \sum_{k=1}^n V( x_k) \leq R n  \Big \} 
% 		\end{equation*}
% 		for some $ R > 0$ and in case of $R=1$, we write $ \mathbb{B}^{n,V} : =  \mathbb{B}_1^{n,V}$. 
% \end{df}

% For some Orlicz function $V: \R \rightarrow [0,\infty)$ and $ \beta \in (- \infty, 0)$, we define the Gibbs measure $\mu_{V,\beta}$ via the Lebesgue density
% \begin{equation*}
% \frac{d \mu_{V,\beta}}{dx}(x):= \frac{e^{\beta V(x)} }{\int_{\R} e^{\beta V(t)} dt}, \quad x \in \R.
% \end{equation*}

% % % % % % % % % % % % % % % % % % % % % % % % %
%\subsection{Basics from large deviation theory and probability}
% % % % % % % % % % % % % % % % % % % % % % % % %

% % % % % % % % % % % % % % % % % % % % % % % % %
\section{Proofs}
% % % % % % % % % % % % % % % % % % % % % % % % %

% % % % % % % % % % % % % % % % % % % % % % % % %
\subsection{Auxiliary Results}
% % % % % % % % % % % % % % % % % % % % % % % % %

% If $V$ and $W$ are Orlicz functions with $\lim_{x\rightarrow \infty} \frac{V(x)}{W(x)} = \infty$, $ \frac{d \mu_{\alpha, \beta}(x)}{d  x}$ is well-defined for any pair $(\alpha, \beta) \in (- \infty,0) \times \R $. 
% In the following, let $V$ and $W$ be Orlicz functions with $ \lim_{x\rightarrow \infty} \frac{V(x)}{W(x)} = \infty$ and we define $ \varphi(\alpha, \beta) = \log Z_{\alpha, \beta}$.

In this Section we work with Orlicz functions $V,W : \R \rightarrow [0, \infty)$ and Gibbs measures $\mu_{\alpha, \beta}$ with distribution given in \eqref{Eq_Def_GibbsMeas}, where $(\alpha, \beta) \in \R^2$ are chosen such that $\mu_{\alpha, \beta}$ exists. We recall that $Z_{\alpha, \beta} = \int_{\R} e^{\alpha V(x) + \beta W(x)} dx$, $\varphi(\alpha, \beta) =  \log Z(\alpha, \beta)$ and define
\begin{align*}
     \mathcal{D}_{\varphi} := \left \{ (\alpha ,\beta) \in \R^2 : \varphi(\alpha, \beta) < \infty \right \}.
\end{align*}
The following lemma gives (partial) characterisations of the set $\mathcal{D}_{\varphi}$, depending on the relation of $V$ and $W$. 
\begin{lem}
\label{Lem_Domain_phi}
$(i)$ If $V$ and $W$ satisfy Assumption \ref{Ass_A}, then 
\begin{equation*}
\left \{ (\alpha, \beta ) \in (- \infty, 0) \times \R : \beta < | \alpha | \liminf_{x \rightarrow \infty} \frac{V(x)}{W(x)}  \right \} \subseteq \mathcal{D}_{\varphi}.
\end{equation*}
$(ii)$ If $V$ and $W$ satisfy Assumption \ref{Ass_B}, then $\mathcal{D}_{\varphi} = (- \infty, 0) \times \R$.
\end{lem}
\begin{proof}
We start with $(i)$, where we assume that $\bar{x} := \liminf_{x \rightarrow \infty} \frac{V(x)}{W(x)} \in (0, \infty)$, since the case where $\bar{x} = \infty $ is treated in $(ii)$. Let $\alpha \in (- \infty, 0)$ be fixed and take some $\beta \in \R$ with $\beta < | \alpha | \liminf_{x \rightarrow \infty} \frac{V(x)}{W(x)}$. We fix $\delta \in (0, \infty)$ and take $x_0 \in (0, \infty)$ such that for all $x \in [x_0, \infty)$, we have $\frac{V(x)}{W(x)} \geq \bar{x} - \delta$. We then get
\begin{align*}
    \int_{x_0}^{\infty}e^{\alpha V(x) + \beta W(x)} dx & = \int_{x_0}^{\infty} \exp \left (\alpha V(x) \left( 1 + \frac{\beta W(x)}{\alpha V(x)} \right ) \right ) dx \\
    & \leq \int_{x_0}^{\infty} \exp \left( \alpha V(x) \left( 1 + \frac{\beta }{\alpha ( \bar{x} - \delta )} \right ) \right) dx \\
    & \leq \int_{x_0}^{\infty} \exp \left( \alpha V(x) \left( 1 - \frac{\bar{x} }{  \bar{x} - \delta } \right ) \right) dx < \infty,
\end{align*}
where the last inequality holds since $ 1 - \frac{\bar{x} }{  \bar{x} - \delta } > 0$, for fixed $\delta $. Thus, as claimed $(\alpha, \beta) \in \mathcal{D}_{\varphi}$. The proof of $(ii)$ works with an analogous argument.
\end{proof}

\begin{lem}
\label{Lem_AuxL1}
Let $V, W : \R \rightarrow [0, \infty)$ be two Orlicz functions. Then the mappings
\begin{equation*}
    \beta \mapsto \mathbb{E}_{\mu_{\alpha, \beta}}\left [ V(X) \right] \quad \text{and} \quad  \alpha \mapsto \mathbb{E}_{\mu_{\alpha, \beta}}\left [ V(X) \right], \quad (\alpha, \beta) \in \mathcal{D}_{\varphi},
\end{equation*}
are well-defined and strictly monotone increasing.
\end{lem}
\begin{proof}
Let $(\alpha, \beta ) \in \mathcal{D}_{\varphi}$, then 
\begin{equation*}
     \frac{\partial }{\partial \alpha } \varphi( \alpha ,\beta ) =  \mathbb{E}_{\mu_{\alpha, \beta}}\left [ V(X) \right]
\end{equation*}
exists (as a derivative of a parameterintegral). We prove in the following that, for fixed $\alpha \in (- \infty, 0)$, $\beta \mapsto \mathbb{E}_{\mu_{\alpha, \beta}}\left [ V(X) \right]$ is strictly increasing. To that end, let $\beta_1, \beta_2 \in \R$ such that $(\alpha, \beta_i) \in \mathcal{D}_{\varphi}$ for $i=1,2$ and let $\beta_1 < \beta_2$. Since $V$ is an Orlicz function (hence $V$ takes only non-negative values) it follows that $Z_{\alpha, \beta_1} < Z_{\alpha, \beta_2}$. Hence, there exists a unique point $ \bar{x} \in (0, \infty) $ such that 
\begin{equation*}
    \frac{d \mu_{\alpha, \beta_1}}{dx} ( \bar{x}) = \frac{e^{\alpha V( \bar{x}) + \beta_1 W(\bar{x}) }}{Z_{\alpha, \beta_1}} = \frac{e^{\alpha V( \bar{x}) + \beta_2 W(\bar{x}) }}{Z_{\alpha, \beta_2}} = \frac{d \mu_{\alpha, \beta_2}}{dx} ( \bar{x}).
\end{equation*}
This implies that 
\begin{align*}
    \frac{e^{\alpha V( x) + \beta_1 W(x) }}{Z_{\alpha, \beta_1}} & > \frac{e^{\alpha V( x) + \beta_2 W(x) }}{Z_{\alpha, \beta_2}} \quad \text{for all } x \in (0, \bar{x}) \\
     \frac{e^{\alpha V( x) + \beta_1 W(x) }}{Z_{\alpha, \beta_1}} & < \frac{e^{\alpha V( x) + \beta_2 W(x) }}{Z_{\alpha, \beta_2}} \quad \text{for all } x \in ( \bar{x}, \infty).
\end{align*}

From that we infer
\begin{align*}
    \mathbb{E}_{\mu_{\alpha, \beta_1}} \left[ V(X) \right]
    & = \int_{0}^{\infty} \mathbb{P}_{\alpha, \beta_1} \left[ V(X) \geq x \right] dx \\
    & = 2 \int_{0}^{\infty} \int_{ V^{-1}(x)}^{\infty} \frac{e^{\alpha V(t) + \beta_1 W(t)}}{Z_{\alpha, \beta_1}} dt dx \\
    &= 2 \int_{0}^{V(\bar{x})} \int_{ V^{-1}(x)}^{\infty} \frac{e^{\alpha V(t) + \beta_1 W(t)}}{Z_{\alpha, \beta_1}} dt dx  + 2 \int_{V(\bar{x})}^{\infty} \int_{V^{-1}(x)}^{\infty} \frac{e^{\alpha V(t) + \beta_1 W(t)}}{Z_{\alpha, \beta_1}} dt dx \\
    & < 2 \int_{0}^{V(\bar{x})} \left [ \frac{1}{2} -  \int_{0}^{ V^{-1}(x)}  \frac{e^{\alpha V(t) + \beta_1 W(t)}}{Z_{\alpha, \beta_1}} dt \right ] dx  + 2 \int_{V(\bar{x})}^{\infty} \int_{V^{-1}(x)}^{\infty} \frac{e^{\alpha V(t) + \beta_2 W(t)}}{Z_{\alpha, \beta_2}} dt dx \\
    & < 2 \int_{0}^{V(\bar{x})} \left [ \frac{1}{2} -  \int_{0}^{ V^{-1}(x)} \frac{e^{\alpha V(t) + \beta_2 W(t)}}{Z_{\alpha, \beta_2}} dt \right ] dx  + 2 \int_{V(\bar{x})}^{\infty} \int_{V^{-1}(x)}^{\infty} \frac{e^{\alpha V(t) + \beta_2 W(t)}}{Z_{\alpha, \beta_2}} dt dx \\
    & = \mathbb{E}_{\mu_{\alpha, \beta_2}} \left [ V(X) \right].
\end{align*}
An analogous argument shows that $ \alpha \mapsto \mathbb{E}_{\mu_{\alpha, \beta}}\left[ V(X) \right]$ is strictly monotone increasing.
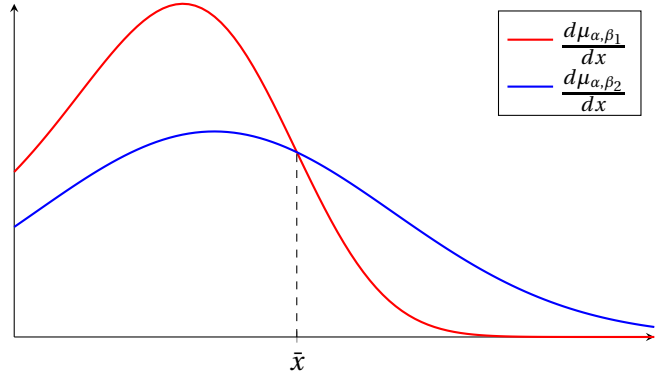
\begin{figure}[h!]
\centering
    %\label{fig:enter-label}
\begin{tikzpicture}
\begin{axis}[
clip=false,
    %xlabel = \(x\),
    %ylabel = \(y\),
    domain = 0:4, 
    samples = 100, 
    axis lines = left,
    width = 10cm,
    height = 6cm,
    ytick=\empty,
   xtick={1.76647}, 
   xticklabels={\(\bar{x}\)}
]
\addplot[red, thick, domain=0:4] {1.5*exp(-0.3*x^3 + x)};
\addplot[blue, thick, domain=0:4] {exp(-0.4*x^2 + x )};  
\legend{\(\frac{d \mu_{\alpha, \beta_1}}{dx}\), \( \frac{d \mu_{\alpha, \beta_2}}{dx}\)}

\end{axis}
\draw[dashed] (2.103*1.76647,0) -- (2.103*1.76647,1.5*1.67917);
\end{tikzpicture}
\caption{The picture suggests that $\mathbb{E}_{\mu_{\alpha, \beta_1}} [V(X)] < \mathbb{E}_{\mu_{\alpha, \beta_2}} [V(X)]$.}    
\end{figure}
\end{proof}

% \begin{lem}
% \label{Lem_AuxL2a}
% Let $R \in (0, \infty)$ and $V,W: \R \rightarrow [0, \infty)$ be two Orlicz functions that satisfy Assumption \ref{Ass_A}
% \end{lem}

\begin{lem}
    \label{Lem_AuxL2}
    Let $R \in (0, \infty)$ and $V,W : \R \rightarrow [0, \infty)$ be two linearly independent Orlicz functions.
    \par{}
    $(i)$ If $V$ and $W$ satisfy Assumption \ref{Ass_A}, then there exists an $\varepsilon \in (0, \infty )$ such that for all $t \in (m, m + \varepsilon)$, we find $(\alpha, \beta) \in (- \infty, 0) \times (0, \infty)$ such that
\begin{equation}
\label{Eq_Gradient}
    \nabla \varphi( \alpha, \beta )  = \left(
\begin{array}{c}
R \\
t \\
\end{array}
\right).
\end{equation}
$(ii)$ If $V$ and $W$ satisfy Assumption \ref{Ass_B}, then, for all $t \in (m, W(V^{-1}(R)))$, we find $(\alpha, \beta) \in (- \infty, 0) \times (0, \infty)$ such that \eqref{Eq_Gradient} holds.
\end{lem}

\begin{proof}
We start with a few general observation before we prove $(i)$ and $(ii)$.
It is easy to see that $\varphi : \mathcal{D}_{\varphi} \rightarrow \R$ is two times continuous differentiable and, for $(\alpha, \beta) \in \mathcal{D}_{\varphi}$, the Hessian $H_{\varphi}(\alpha, \beta)$ of $\varphi$ satisfies
\begin{equation*}
H_{\varphi}(\alpha, \beta) = \Cov_{\mu_{\alpha, \beta}} \left[  \left( V(X), W(X) \right) \right].
\end{equation*}
Since $ H_{\varphi}(\alpha, \beta)$ is the covariance matrix of the vector $\left(V(X), W(X) \right)$, where $X \sim \mu_{\alpha, \beta}$ and $V$ and $W$ are linearly independent on $ \supp(X) = \R $, we can deduce that $H_{\varphi}(\alpha, \beta)$ is a positive definite matrix. We define $\alpha^{*} \in (- \infty, 0)$ to be the unique point, where (see, e.g., Section 1.1. in \cite{AP2022}) 
\begin{equation*}    
\mathbb{E}_{\mu_{\alpha^{*},0}}[V(X)] = R.
\end{equation*}
We are now going to prove $(i)$. By Lemma \ref{Lem_Domain_phi}, there is a small neighborhood (in $\R^2$) around $\alpha_*$ which is completely contained in $\mathcal{D}_{\varphi}$. By Lemma \ref{Lem_AuxL1}, the mapping $\beta \mapsto  \mathbb{E}_{\mu_{\alpha^*, \beta}} \left[ V(X) \right]$ is strictly monotone increasing and hence, by continuity, there exists a $\delta_1 \in (0, \infty)$ such that for all $\alpha \in (\alpha^* - \delta_1, \alpha^*)$, we find a $\beta' \in (0, \infty )$ such that
\begin{equation*}
\mathbb{E}_{\mu_{\alpha, \beta'}} \left[ V(X) \right] = R.
\end{equation*}
Since $\beta \mapsto \mathbb{E}_{\mu_{\alpha, \beta}} \left[ V(X) \right]  $ is strictly mononote increasing, we see that $\beta'$ is unique. This defines a mapping $\alpha \mapsto \beta(\alpha)$ for $\alpha \in (\alpha^* - \delta_1, \alpha^*)$ with the property
\begin{equation*}
 \mathbb{E}_{\mu_{\alpha, \beta(\alpha)}} \left[ V(X) \right] = R.   
\end{equation*}
The function $\alpha \mapsto \beta(\alpha)$ is continuous and by employing Lemma \ref{Lem_AuxL2}, $ \beta(\alpha)$ is strictly monotone decreasing on $ (\alpha^* - \delta_1, \alpha^*)$. By the Lebesgue differentiation theorem, it follows that $ \beta : (\alpha^* - \delta_1, \alpha^*) \rightarrow (0, \infty)$ is differentiable almost everywhere. Since, by construction, $\mathbb{E}_{\mu_{\alpha, \beta(\alpha)}} \left[ V(X) \right] \equiv R$, we have that $\frac{\partial}{\partial \alpha } \mathbb{E}_{\mu_{\alpha, \beta(\alpha)}} \left[ V(X) \right] \equiv 0$ and in every point $\alpha \in ( \alpha^* - \delta_1, \alpha^*)$ where $\beta(\alpha)$ is differentiable, we get by the chain rule
\begin{equation*}
    0 = \frac{\partial}{\partial \alpha } \mathbb{E}_{\mu_{\alpha, \beta(\alpha)}} \left[ V(X) \right] = \mathbb{V}_{\mu_{\alpha, \beta(\alpha)}} \left[ V(X) \right ] + \beta'( \alpha) \Cov_{\mu_{\alpha, \beta(\alpha)}} \left[ V(X), W(X) \right]
\end{equation*}
We recall that $H_{\varphi}( \alpha, \beta(\alpha) )$ is a positive definite matrix which implies that
\begin{align*}
    0 < \left( 1, \beta'(\alpha) \right) H_{\varphi} ( \alpha, \beta(\alpha) ) \left( 
    \begin{array}{cc}
        1  \\
         \beta'(\alpha)  
    \end{array} \right) 
    & =\beta'( \alpha) \left( \Cov_{\mu_{\alpha, \beta(\alpha)}} \left[ V(X), W(X) \right] + \beta'(\alpha) \mathbb{V}_{\mu_{\alpha, \beta(\alpha)}} \left[ W(X) \right ] \right) \\
    & = \beta'(\alpha) \left(  \frac{\partial}{\partial \alpha } \mathbb{E}_{\mu_{\alpha, \beta( \alpha)}} \left[ W(X) \right] \right).
\end{align*}
Since $\beta'(\alpha) \in (- \infty, 0)$, we have $ \frac{\partial}{\partial \alpha } \mathbb{E}_{\mu_{\alpha, \beta( \alpha)}} \left[ W(X) \right] \in (- \infty, 0)$. This means, in almost every point $ \alpha \in (\alpha^* - \delta_1, \alpha^*) $, $ \frac{\partial}{\partial \alpha } \mathbb{E}_{\mu_{\alpha, \beta( \alpha)}} \left[ W(X) \right]$ exists and is negative. Thus, the mapping 
\begin{equation*}
    \alpha \mapsto \mathbb{E}_{\mu_{\alpha, \beta(\alpha)}} \left[ W(X) \right]
\end{equation*}
is continous and strictly monotone decreasing on $(\alpha^* - \delta_1, \alpha^*)$. Since $m = \mathbb{E}_{\mu_{\alpha^*, 0}} \left[ W(X) \right]$, it follows that
\begin{equation*}
    \varepsilon := \sup_{ \alpha \in ( \alpha^* - \delta_1, \alpha^*)}\mathbb{E}_{\mu_{\alpha, \beta(\alpha)}} \left[ W(X) \right] - m \in (0, \infty), 
\end{equation*}
and for every $t \in (m, m+ \varepsilon)$, we find a pair $(\alpha, \beta) \in (- \infty, 0) \times (0, \infty)$ such that
\begin{equation*}
    \nabla \varphi(\alpha, \beta) = 
    \left( \begin{array}{c}
         \mathbb{E}_{\mu_{\alpha, \beta}} \left[ V(X) \right]  \\
         \mathbb{E}_{\mu_{\alpha, \beta}} \left[ W(X) \right]
    \end{array} \right)
    = \left(
    \begin{array}{c}
         R  \\
         t
    \end{array} \right),
\end{equation*}
which proves $(i)$.
\par{}
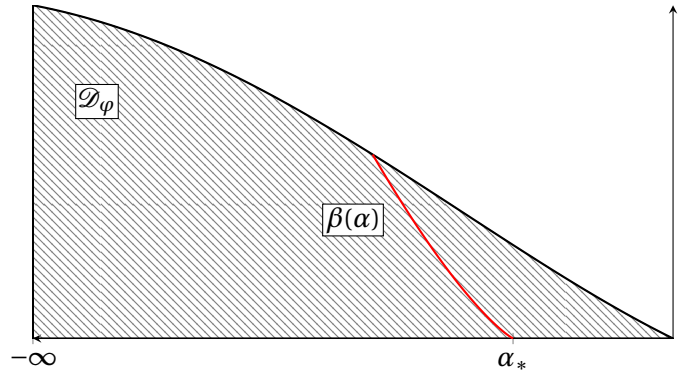
\begin{figure}[ht!]
    \centering
\begin{tikzpicture}
\begin{axis}[
    domain = -2:2,
    samples = 200,
    axis x line = bottom,
    x axis line style={stealth-},
    axis y line = right, 
    width = 10cm,
    height = 6cm,
   ytick=\empty,
   yticklabel={\( \R_+ \)},
   xtick=\empty, 
   xtick={-2},
   xticklabel={\( - \infty \)},
   extra x ticks={-0.5}, 
   extra x tick label={\(  \alpha_* \)},
]
\addplot[black, dashed, domain=-2:0] { {0.2*x^3 + x^2 + 2*abs(sin(deg(x)))} };
\addplot[pattern=north west lines, line width=0.8pt, pattern color=gray, domain=-2:0] { {0.2*x^3 + x^2 + 2*abs(sin(deg(x)))}} \closedcycle;

\draw (20, 300) node[rectangle, draw, fill=white, inner sep=1pt] {$\mathcal{D}_{\varphi}$}; 
\addplot[red, thick, domain=-0.938487:-0.5] {5*(-x-0.5)^(1.5) + 2*abs(x + 0.5)};
\draw (100, 150) node[rectangle, draw, fill=white, inner sep=1pt] {$\beta(\alpha)$};
\end{axis}
\end{tikzpicture}
\caption{In $(i)$, we are able to construct $\beta(\alpha)$ locally to the left of $\alpha_*$.}
\end{figure}
Now we are going to prove $(ii)$. By Lemma \ref{Lem_Domain_phi}, it follows that $\mathcal{D}_{\varphi} = ( - \infty, 0) \times \R$, hence for every $\alpha \in (- \infty, \alpha^*)$, we find an unique $\beta = \beta(\alpha) \in (0, \infty)$ such that
\begin{equation*}
    \mathbb{E}_{\mu_{\alpha, \beta(\alpha)}} \left[ V(X)\right] = R.
\end{equation*}
The mapping $\alpha \mapsto \beta(\alpha)$ is strictly monotone increasing (hence differentiable almost everywhere) and analogously to the first part of this prove, one has that 
\begin{equation*}
    \alpha \mapsto \mathbb{E}_{\mu_{\alpha, \beta(\alpha)}} \left[ W(X) \right]
\end{equation*}
is continuous and strictly monotone decreasing. We define
\begin{equation*}
    x := \sup_{\alpha \in (- \infty, \alpha^*)} \mathbb{E}_{\mu_{\alpha, \beta(\alpha)}} \left[ W(X) \right] = \lim_{\alpha \rightarrow - \infty } \mathbb{E}_{\mu_{\alpha, \beta(\alpha)}} \left[ W(X) \right].
\end{equation*}
By the continuity of $\alpha \mapsto \mathbb{E}_{\mu_{\alpha, \beta(\alpha)}} \left[ W(X) \right]$, for all $t \in (m, x)$ we find $(\alpha, \beta ) \in (- \infty, 0) \times (0, \infty) $ with
\begin{equation*}
\nabla \varphi(\alpha, \beta) = 
    \left( \begin{array}{c}
         \mathbb{E}_{\mu_{\alpha, \beta}} \left[ V(X) \right]  \\
         \mathbb{E}_{\mu_{\alpha, \beta}} \left[ W(X) \right]
    \end{array} \right)
    = \left(
    \begin{array}{c}
         R  \\
         t
    \end{array} \right).
\end{equation*}
In the rest of this proof, we will show that $x = W(V^{-1}(R))$. Let $\mathcal{M} = (\mu_{\alpha, \beta( \alpha)})_{\alpha \in (- \infty, \alpha^*)}$, then $\mathcal{M}$ is a tight family of probability measures (otherwise $R \not\equiv \mathbb{E}_{\mu_{\alpha, \beta(\alpha)}} \left[ V(X) \right] $), thus every sub-family of $\mathcal{M}$ is also tight.  
For any sequence $(\alpha_n)_{n \in \N}$ with $\lim_{n \rightarrow \infty} \alpha_n = - \infty$, we have that 
\begin{equation*}
    x = \lim_{n \rightarrow \infty}  \mathbb{E}_{\mu_{\alpha_n, \beta(\alpha_n)}} \left[ W(X) \right].
\end{equation*}
Assume there exists a sequence $(\alpha_n)_{n \in \N}$ tending to $- \infty$ with 
\begin{equation*}
    \lim_{n \rightarrow \infty} \mathbb{V}_{\mu_{\alpha_n, \beta(\alpha_n)}} \left[ V(X) \right] = 0.
\end{equation*}
Since $(\mu_{\alpha_n, \beta(\alpha_n)})_{n \in \N}$ is tight, we find a weakly convergent subsequence $(\mu_{\alpha_{n_k}, \beta(\alpha_{n_k})})_{k \in \N} $ with (weak) limit $\mu_{\infty}$. We trivially have that 
\begin{equation*}
    \mathbb{E}_{\mu_{\infty}} \left[ V(X) \right] =R.
\end{equation*}
But, since $ \mathbb{V}_{\mu_{\infty}} \left[ V(X) \right] = 0$ and by the symmetry of $\mu_{\infty} $, it follows that
\begin{equation*}
    \mu_{\infty} = \frac{1}{2} \delta_{-V^{-1}(R)} +  \frac{1}{2} \delta_{V^{-1}(R)}.
\end{equation*}
This gives us that $x = W(V^{-1}(R))$. Now consider the case where $\inf_{\alpha \in (- \infty, \alpha^*)} \mathbb{V}_{\mu_{\alpha, \beta(\alpha)}} \left[ V(X) \right] \in (0, \infty)$. We know that $ x \in (0, \infty)$, otherwise 
\begin{equation*}
    \lim_{\alpha \rightarrow -\infty } \mathbb{E}_{\mu_{\alpha, \beta(\alpha)}} \left[ W(X) \right] = \infty,
\end{equation*}
which is a contradiction to $ R = \lim_{\alpha \rightarrow - \infty} \mathbb{E}_{\mu_{\alpha, \beta(\alpha)}} \left[ V(X) \right]$. Let $A \subseteq (- \infty, \alpha^*)$ be the set where $\beta'$ 

(and hence also $\frac{\partial}{\partial \alpha } \mathbb{E}_{\mu_{\alpha, \beta(\alpha)}} \left[ W(X) \right]$) exists. We claim that there exists a subsequence $(\alpha_n)_{n \in \N} \subseteq A$ such that
\begin{equation*}
  \lim_{n \rightarrow \infty} \frac{\partial}{\partial \alpha } \mathbb{E}_{\mu_{\alpha_n, \beta(\alpha_n)}} \left[ W(X) \right] = 0. 
\end{equation*}
If such a sequence $(\alpha_n)_{n \in \N} \subseteq A$ would not exist, then there is an $\varepsilon \in (0, \infty)$ and $\alpha' = \alpha'(\varepsilon) \in (- \infty, \alpha^*)$ such that for all $\alpha \in (- \infty, \alpha')$ we have
\begin{equation*}
   \frac{\partial}{\partial \alpha } \mathbb{E}_{\mu_{\alpha, \beta(\alpha)}} \left[ W(X) \right] \leq - \varepsilon.
\end{equation*}
By integrating over the infinite interval $(- \infty, \alpha')$ we would now get $x = \infty$, which is a contradiction. Thus, we can assume that there exists a sequence $(\alpha_n)_{n \in \N} \subseteq A$ with 
\begin{equation}
\label{Eq_Derivative-EW}
  \lim_{n \rightarrow \infty} \frac{\partial}{\partial \alpha } \mathbb{E}_{\mu_{\alpha_n, \beta(\alpha_n)}} \left[ W(X) \right] = 0. 
\end{equation}
We claim that, along this sequence $(\alpha_n)_{n \in \N}$, we have $\sup_{n \in \N}  \frac{|\beta(\alpha_n)|}{\alpha_n}  \in (0, \infty)$. If this was not the case, then, along a suitable subsequence $(\alpha_{n_k})_{k \in \N} \subseteq (\alpha_{n})_{n \in \N}$, we would have that $\lim_{k \rightarrow \infty} \frac{| \beta(\alpha_{n_k}) | }{\alpha_{n_k}}= \infty $. Let $M \in (0, \infty)$ be some large number and denote $x_k \in (0, \infty)$ to be the maximizer of $x \mapsto \alpha_{n_k} V(x) + \beta(\alpha_{n_k}) W(x)$ on $(0,M)$ and let $y_k \in (0, \infty) $ be the minimizer of $y \mapsto \alpha_{n_k} V(y) + \beta( \alpha_{n_k}) W(y)$ on $(2M,3M)$. Then, since $V$ and $W$ are Orlicz functions, there exist constants $c, C\in (0, \infty)$ such that $ V(y_k) - V(x_k) \leq C$ and $W(y_k) - W(x_k) \geq c$ for all $k \in \N$. This leads to
\begin{align*}
    \mu_{\alpha_{n_k}, \beta( \alpha_{n_k})}  \left[(0,M)  \right]
    &= \frac{2\int_{0}^M e^{\alpha_{n_k}V(x) + \beta( \alpha_{n_k}) W(x)} dx }{Z_{\alpha_{n_k}}} \\
    & \leq \frac{2 M e^{\alpha_{n_k}V(x_k) + \beta( \alpha_{n_k}) W(x_k)} }{Z_{\alpha_{n_k}}} \\
    & \leq e^{ - \alpha_{n_k} \left( V(y_k) - V(x_k) \right) - \beta( \alpha_{n_k}) \left( W(y_k) - W(x_k) \right)} \\
    & \leq e^{ - \alpha_{n_k} C - \beta( \alpha_{n_k}) c } \stackrel{k \rightarrow \infty}{\longrightarrow} 0.
\end{align*}
This means, for every fixed $M$, the $\mu_{\alpha_{n_k}, \beta( \alpha_{n_k})}$-probability of $(0,M)$ tends to $0$ as $k \rightarrow \infty$. This contradicts the construction of $\beta(\alpha_{n_k})$, where we have
\begin{equation*}
    R= \mathbb{E}_{\mu_{\alpha_{n_k}, \beta(\alpha_{n_k})}} \left[ V(X) \right].
\end{equation*}
Thus, $ \left(  \frac{\beta(\alpha_n)}{\alpha_n} \right)_{n \in \N} $ is a bounded sequence and we find a convergent subsequence $ (\gamma_n^{(1)})_{n \in \N} \subseteq \left( \alpha_n \right)_{n \in \N}$  with limit $- \lambda \in (- \infty, 0)$. Further, by de L'Hospital's rule, we have that $- \lambda = \lim_{n \rightarrow \infty}  \beta'(\alpha_n)$.

Since $\mathcal{M}$ is tight, we find another subsequence $(\gamma_n^{(2)})_{n \in \N} \subseteq (\gamma_n^{(1)})_{n \in \N}$ and a distribution $\mu_{\infty}$ such that $\mu_{\gamma_n^{(2)}, \beta(\gamma_n^{(2)})} \longrightarrow \mu_{\infty}$ as $n \rightarrow \infty$ (we understand this limit as weak limit of distributions). 
We claim that for all $x \in \R$, we have $V(x) \geq \lambda W(x)$. Assume this was not the case, i.e., we find a small interval $I =(y- \delta, y + \delta) \subseteq \R$ and some $\varepsilon \in (0, \infty)$ such that for all $x \in (y- \delta, y + \delta)$
\begin{equation*}
  V(x) - \lambda W(x) \leq - \varepsilon. 
\end{equation*}
Then (since $\frac{\beta(\gamma_n^{(2)})}{\gamma_n^{(2)}} \leq - \lambda + \frac{\varepsilon}{2}\left( \sup_{x \in I} W(x) \right)^{-1}$, if $n$ is sufficiently large)
\begin{align*}
    Z_{\gamma_n^{(2)}, \beta(\gamma_n^{(2)})}& = \int_{\R} e^{\gamma_n^{(2)}V(x) + \beta(\gamma_n^{(2)})W(x)} dx \\
    & = \int_{\R} \exp \left( \gamma_n^{(2)} \left( V(x) + \frac{\beta(\gamma_n^{(2)})}{\gamma_n^{(2)}} W(x)\right) \right) dx \\
    & \geq \int_{I} \exp \left( \gamma_n^{(2)} \left( V(x) + \frac{\beta(\gamma_n^{(2)})}{\gamma_n^{(2)}} W(x)\right) \right) dx \\
    & \geq \int_{I} \exp \left( - \frac{\varepsilon \gamma_n^{(2)}}{2}  \right) dx \longrightarrow \infty, \quad n \rightarrow \infty.
\end{align*}
From that we deduce 
\begin{equation*}
    \mu_{\infty} \left[  x \in \R : V(x) - \lambda W(x) \leq 0 \right] = 1 \quad \text{and} \quad \mu_{\infty} \left[  x \in \R : V(x) - \lambda W(x) \leq - \varepsilon \right] > 0.
\end{equation*}
This means, for a random variable $X \sim \mu_{\infty}$, we have $V(X) \leq \lambda W(X)$ almost surely and with positive probability we have $ V(X) < \lambda W(X)$. By the very construction of $(\alpha_n)_{n \in \N}$, of which $(\gamma_n^{(2)})_{n \in \N}$ is a subsequence, we have by \eqref{Eq_Derivative-EW}
\begin{equation*}
 0 =   \Cov_{\mu_{\infty}} \left[ V(X), W(X) \right] - \lambda \mathbb{V}_{\mu_{\infty}} \left[ W(X) \right ].
\end{equation*}
Recall that this quantity is a (non-trivial) quadratic form of the covariance matrix $C \in \R^{2 \times 2}$ of $\left(V(X), W(X) \right)$. This means $V(X)$ and $W(X)$ are linearly dependent, or in other words, there exists a $\nu \in (0, \infty)$ such that $V(X) = \nu W(X)$ almost surely (although we assumed that $V$ and $W$ are linearly independent on $\R$, this does not necessarily mean that $V$ and $W$ are linearly independent on $\supp(X)$). Since $V(X) \leq \lambda W(X)$ almost surely, we thus have $\nu \leq \lambda$. This leads to
\begin{align*}
 0 & =   \Cov_{\mu_{\infty}} \left[ V(X), W(X) \right] - \lambda \mathbb{V}_{\mu_{\infty}} \left[ W(X) \right ] \\
 & =  \frac{1}{\nu} \mathbb{V}_{\mu_{\infty}} \left[ V(X) \right ] - \frac{\lambda}{\nu^2} \mathbb{V}_{\mu_{\infty}} \left[ V(X) \right ].
\end{align*}
Since by assumption $\mathbb{V}_{\mu_{\infty}} \left[ V(X) \right ] \in (0, \infty)$, we get $\nu = \lambda$. But this is a contradiction, since, on the one hand, we have $V(X) = \lambda W(X)$ almost surely, and on the other hand, $V(X) < \lambda W(X)$ with positive probability. This means, $\lambda $ has the property that $V(x) \geq \lambda W(x)$ for all $x \in \R$. In particular, it is true that
\begin{equation*}
    \lambda \leq \frac{R}{W(V^{-1}(R))}.
\end{equation*}
We now get (using that $\mathbb{E}_{\mu_{\infty}} \left[ V(X) \right] = R $ and $W(X) = \frac{1}{\lambda} V(X)$)
\begin{align*}
    x &= \mathbb{E}_{\mu_{\infty}} \left[ W(X) \right] = \frac{1}{\lambda} \mathbb{E}_{\mu_{\infty}} \left[ V(X) \right] \geq W(V^{-1}(R)).
\end{align*}
This finishes the proof.
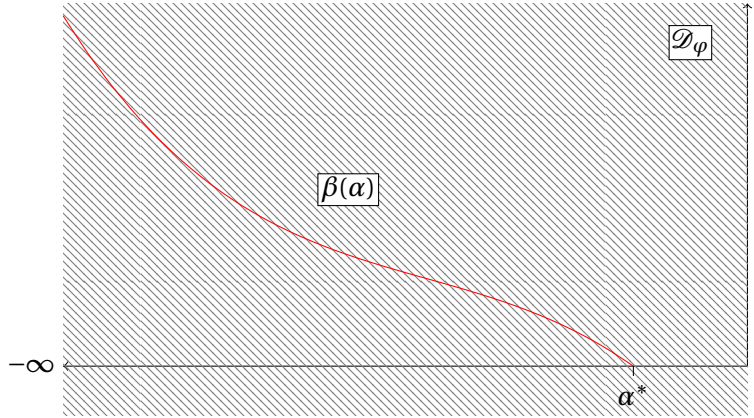
\begin{figure}[ht!]
    \centering
\begin{tikzpicture}[xscale=1.5, yscale=1.3]
    % Koordinatensystem
    \draw[<-] (-6,0) -- (0,0) node[left] {};
    \draw[->] (0,0) -- (0,3.7) node[above] {};

    \fill[ pattern=north west lines, pattern color=gray] (-6,-0.5) -- (0,-0.5) -- (0,-0.5) -- (0,3.7) -- (0,3.7) -- (-6,3.7) -- cycle;
    % Funktion
    \draw[domain=-6:-1,smooth,variable=\x,red] plot ({\x},{0.55*(-0.092*(\x+1)^(3) - 0.5*(\x+1.5)^(2) +0.5^3 -(\x +1) )} ) node[right] {};
   
   \draw (-0.5, 3.3) node[rectangle, draw, fill=white, inner sep=1pt] {$\mathcal{D}_{\varphi}$};
    % Achsenbeschriftungen
    \foreach \x in {-1}
        \draw (\x,0) -- (\x,-0.1) node[below] {$\alpha^*$};
        \draw (-3.5, 1.8) node[rectangle, draw, fill=white, inner sep=1pt] {$\beta(\alpha)$};
        
        % \draw (-6, 3) node[rectangle, draw, fill=white, inner sep=1pt] {$\mathcal{D}_{\varphi}$}; 
        \draw (-6,0) node[left] {$-\infty$};
\end{tikzpicture}
\caption{In $(ii)$, we construct $\beta(\alpha)$ for all $\alpha \in (-\infty,\alpha^*)$. The gray-shaded area indicates the region $\mathcal{D}_{\varphi}$.}
\end{figure}
\end{proof}
\begin{rem}
    \label{Rem_Cond_thm}
    When looking at the proof of Lemma \ref{Lem_AuxL2}, one is tempted to think that assuming
    \begin{equation*}
        \liminf_{x \rightarrow \infty} \frac{V(x)}{W(x)} > 0
    \end{equation*}
    should be sufficient to push $\varepsilon $ as far as possible, i.e., $\varepsilon = W(V^{-1}(R)) -m$. But this is not the case! If we set $V(x):= x^4 + x^2$ and $W(x):=x^4$, then $V$ and $W$ are linearly independent and we have
    \begin{equation*}
        \liminf_{x \rightarrow \infty} \frac{V(x)}{W(x)} = 1 > 0.
    \end{equation*}
    It is clear that for $R=10^3$ and for $\alpha = -1$, we have
    \begin{equation*}
        \sup_{\beta: (\alpha, \beta) \in \mathcal{D}_{\varphi}} \mathbb{E}_{\mu_{-1,\beta }}\left[V(X) \right] = \mathbb{E}_{\mu_{-1, 1}}\left[V(X) \right] < R.
    \end{equation*}
    Thus, the construction of $\beta(\alpha)$ as in the proof of Lemma \ref{Lem_AuxL3} $(ii)$ is possible only locally. We further have
    \begin{equation*}
        \lim_{\alpha \rightarrow - \infty} \sup_{\beta: (\alpha, \beta) \in \mathcal{D}_{\varphi}} \mathbb{E}_{\mu_{\alpha,\beta }}\left[W(X) \right] = \lim_{\alpha \rightarrow - \infty} \mathbb{E}_{\mu_{\alpha, | \alpha | }}\left[W(X) \right] = 0,
    \end{equation*}
    destroying the hope of maximizing $\varepsilon$ by letting $\alpha$ tend to $- \infty$. A better understanding of the image of $\nabla^{-1} (\varphi) $ would help to characterize the precise value of $\varepsilon$ in both, Theorem  \ref{ThmSLD} and Theorem \ref{ThmSLD_thinshell}.
\end{rem}
Heuristically speaking, the goal of Lemma \ref{Lem_AuxL2} was to maximize $ \mathbb{E}_{\mu_{\alpha, \beta}} \left[ W(X) \right]$ under the constraint $\mathbb{E}_{\mu_{\alpha, \beta}} \left[ V(X) \right] \equiv R$. It turned out that this can be achieved by letting $\alpha $ tend to $- \infty$ and choosing $\beta = \beta( \alpha)$ such that $ \mathbb{E}_{\mu_{\alpha, \beta}} \left[ V(X) \right] \equiv R$. In the proof of Theorem \ref{ThmSLD_thinshell}, we need to choose a pair $(\alpha, \beta) \in \R^2$ such that 
\begin{equation*}
    \mathbb{E}_{\mu_{\alpha, \beta}} \left[ V(X) \right] \equiv R \quad \text{and} \quad \mathbb{E}_{\mu_{\alpha, \beta}} \left[ W(X) \right] \text{ is "small"}.
\end{equation*}
This is achieved through the following Lemma.
\begin{lem}
\label{Lem_Aux4}
Let $V,W: \R \rightarrow [0, \infty)$ be two linearly independent Orlicz functions that satisfy Assumption \ref{Ass_A} and let $R \in (0, \infty)$. Then, there exists an $\varepsilon \in (0, \infty)$ such that for all $ t \in (m- \varepsilon, m)$, we find a pair $(\alpha, \beta) \in (-\infty, 0)^2$ such that 
\begin{equation}
\label{Eq_nabla_phi}
\nabla \varphi( \alpha, \beta) = \left( \begin{array}{c}
      R \\
      t
\end{array} \right).
\end{equation}
\end{lem}
\begin{proof}
Let $\alpha_* \in (- \infty,0)$ be the unique solution to
\begin{equation*}    \mathbb{E}_{\mu_{\alpha_*, 0}}\left[V(X) \right] = R.
\end{equation*}
On the other hand, analogously, we can find a unique $\beta_* \in (- \infty, 0)$ such that 
\begin{equation*}
\mathbb{E}_{\mu_{0, \beta_*}}\left[V(X) \right] = R.    
\end{equation*}
Thus, for any $\alpha \in (\alpha_*, 0)$, we can find a unique $\beta= \beta( \alpha) \in (- \infty, 0)$ such that
\begin{equation*}
    \alpha \mapsto \mathbb{E}_{\mu_{0, \beta_*}}\left[V(X) \right] \equiv R.
\end{equation*}
\begin{figure}[h!]
    \centering
\begin{tikzpicture}[xscale=1.3, yscale=1.1]
    % Koordinatensystem
    \draw[<-] (-6,0) -- (0,0) node[left] {};
    \draw[->] (0,0) -- (0,-4) node[above] {};

    \fill[ pattern=north west lines, pattern color=gray] (-6,0) -- (0,0) -- (0,-4) -- (-6,-4) -- cycle;
    % Funktion
    % \draw[domain=-6:-1,smooth,variable=\x,red] plot ({\x},{0.55*(-0.092*(\x+1)^(3) - 0.5*(\x+1.5)^(2) +0.5^3 -(\x +1) )} ) node[right] {};

    % Achsenbeschriftungen
    \foreach \x in {-5}
        \draw (\x,0) -- (\x,-0.1) node[below] {$\alpha^*$};
        \draw (0,-3) -- (0.1,-3)
        node[right] {$\beta^*$};
        \draw (-1.7, -1.8) node[rectangle, draw, fill=white, inner sep=1pt] {$\beta(\alpha)$};
         \draw (-5.5, -3.5) node[rectangle, draw, fill=white, inner sep=1pt] {$\mathcal{D}_{\varphi}$};
        \draw[domain=-5:0,smooth,variable=\x,red] plot ({\x},{ -3/25*((\x+5)^2)}) node[right] {};
        % \draw (-6, 3) node[rectangle, draw, fill=white, inner sep=1pt] {$\mathcal{D}_{\varphi}$}; 
        \draw (-6,0) node[left] {$-\infty$};
        %\addplot[black, dashed, domain=-2:0] { {0.2*x^3 + x^2 + 2*abs(sin(deg(x)))} };
        \fill [pattern=north west lines, pattern color=gray, domain=-6:0, variable=\x]
      (-6, 0)
      -- plot ({\x}, {0.15*(0.01*abs(\x)^3 + 0.1*(\x)^2 + 0.8*abs(sin(deg(\x))))})
      -- cycle;

       \draw[domain=-6:0,smooth,variable=\x,black] plot ({\x},{0.15*(0.01*abs(\x)^3 + 0.1*(\x)^2 + 0.8*abs(sin(deg(\x))))}) node[right] {};

       \fill [pattern=north west lines, pattern color=gray, domain=-4:0, variable=\x, rotate=90]
      (-4, 0)
      -- plot ({\x}, {-0.4*(-0.05*abs(\x)^3 + 0.3*(\x)^2 )})
      -- cycle;
      
      \draw[domain=-4:0,smooth,variable=\x,black, rotate=90] plot ({\x},{-0.4*(-0.05*abs(\x)^3 + 0.3*(\x)^2 )}) node[right] {};
      %  \draw[pattern=north west lines, line width=0.8pt, pattern color=gray, domain=-5:0, variable=\x] { {0.01*(0.2*\x^3 + \x^2 + 2*abs(sin(deg(\x))))}};

\end{tikzpicture}
\caption{In Lemma \ref{Lem_Aux4}, the construction of $\beta(\alpha)$ changes slightly compared to Lemma \ref{Lem_AuxL2}.}
\end{figure}
\par{}
Analogously to the proof of Lemma \ref{Lem_AuxL2}, we have that $ \beta(\alpha)$ is strictly decreasing and hence almost surely differentiable. Moreover, it holds that 
\begin{equation*}
    \alpha \mapsto \mathbb{E}_{\mu_{\alpha, \beta(\alpha)}} \left[ W(X) \right]
\end{equation*}
is strictly decreasing on $(\alpha_*, 0)$ and thus, by defining
\begin{equation*}
    \varepsilon := m - \mathbb{E}_{\mu_{0, \beta_*}} \left[ W(X) \right] \in (0, \infty),
\end{equation*}
the claim follows.
% \begin{figure}[h!]
%     \centering
% \begin{tikzpicture}
% \begin{axis}[
%     domain = -2:2,
%     samples = 200,
%     axis x line = bottom,
%     x axis line style={stealth-},
%     axis y line = right, 
%     width = 10cm,
%     height = 6cm,
%    ytick=\empty,
%    xtick=\empty, 
%    xtick={-2},
%    xticklabel={\( - \infty \)},
%    extra x ticks={-0.5}, 
%    extra x tick label={\(  \alpha_* \)},
% ]
% \addplot[black, dashed, domain=-2:0] { {0.2*x^3 + x^2 + 2*abs(sin(deg(x)))} };
% \addplot[pattern=north west lines, line width=0.8pt, pattern color=gray, domain=-2:0] { {0.2*x^3 + x^2 + 2*abs(sin(deg(x)))}} 
% \closedcycle;
% \draw (20, 300) node[rectangle, draw, fill=white, inner sep=1pt] {$\mathcal{D}_{\varphi}$}; 
% \addplot[red, thick, domain=-2:0] {-2- x + 2*sin(x)};
% \draw (100, 150) node[rectangle, draw, fill=white, inner sep=1pt] {$\beta(\alpha)$};
% \end{axis}
% \end{tikzpicture}
% \caption{Picture of $\beta (\alpha)$ in the segment $(\alpha_*, 0)$. }
% \end{figure}
\end{proof}

\subsection{Establishing an Edgeworth expansion}
\label{Section:Edgeworth}

In this subsection we prove that, for linearly independent $V,W$, the sequence of random variables
\begin{equation}
\label{Eq_rnd_vec}
    \frac{1}{\sqrt{n}} \sum_{i=1}^n \left(V(X_i), W(X_i) \right),
\end{equation}
where $(X_i)_{i \in \N}$ is an iid sequence with $X_1 \sim \mu_{\alpha, \beta}$, satisfies an Edgeworth expansion. To that end, we need to establish a Cram\'{e}r-type bound for the characteristic function $\varphi : \R^2 \rightarrow \mathbb{C}$ of \eqref{Eq_rnd_vec} given by 
\begin{equation*}
    \varphi(t,s) := \mathbb{E}_{\mu_{\alpha, \beta}} \left[ \exp \left( i t V(X) + i s W(X) \right) \right], \quad (t,s) \in \R^2.
\end{equation*}
In the proof of Lemma \ref{Lem_AuxL3} we show that
\begin{equation}
\label{Eq_Char_fct_Cancelation}
    \limsup_{||(t,s)||_2 \rightarrow \infty} \left | \varphi ( t,s) \right | =\limsup_{||(t,s)||_2 \rightarrow \infty} \left | \frac{1}{Z_{\alpha, \beta }} \int_{\R} e^{\alpha V(x) + \beta W(x) } e^{ i t V(x) + i s W(x)} dx \right |   < 1.
\end{equation}
In the parlance of probabilistic number theory, \eqref{Eq_Char_fct_Cancelation} is satisfied if the oscillatory integral given by the characteristic function $\varphi $ shows at least a very small cancellation. Showing \eqref{Eq_Char_fct_Cancelation} is not straight forward, since our sequence in \eqref{Eq_rnd_vec} is a sequence of random vectors in $\R^2$. In the classical theory (see, e.g., \cite[Section 2]{bhattacharya2010}), we would like to see that the distribution of
\begin{equation*}
    \left( V(X_1), W(X_1) \right), \quad X_1 \sim \mu_{\alpha, \beta},
\end{equation*}
has a non-degenerate (w.r.t. the $2$-dimensional Lebesgue measure) component. Since this particular sequence is concentrated on the $1$-dimensional manifold
\begin{equation*}
    \mathcal{M}:= \left \{  \left (V(x), W(x) \right)  | x \in \R  \right \},
\end{equation*}
the distribution of the random vector in \eqref{Eq_rnd_vec} is degenerated in the $2$-dimensional sense. So we need to tackle the limit in \eqref{Eq_Char_fct_Cancelation} directly. 
The classical tool to deal with such integrals is Van der Corput's lemma (see Proposition \ref{Prop_vanderCorput}). In our case, we are dealing with a phase function of the form $tV+ sW$, where $(t,s) \in \R^2$ and $V,W$ are either Orlicz functions or $V(x) = |x|^p$ and $W(x) = |x|^q$ with $p,q \in (0,1)$ and $p \neq q$. In many cases, the function $x \mapsto tV(x)+ sW(x)$ lacks the necessary smoothness and may also have many inflection points over $\R$. Thus, we are not able to directly apply Van der Corput's lemma. Our idea is to find a family of intervals $J = J(t,s) \subseteq \R$ such that the function $ t V + s W$ restricted on $J$ is $\mathcal{C}^2$ and is strictly monotone and has no inflection points on $J$. On such an interval we can apply Van der Corput's lemma to see at least some amount of cancellation in \eqref{Eq_Char_fct_Cancelation}. In order for everything to work out nicely, we need to impose more regularity on $V$ and $W$, which is the reason for our notion of an Orlicz function.

% \begin{Ass}
%     \label{Ass_C}
%     We say two Orlicz functions $V,W: \R \rightarrow [0, \infty)$ satisfy Assumption \ref{Ass_C}, if they are linearly independent and $\mathcal{C}^2$ in all but countably many points that do not accumulate anywhere. 
% \end{Ass}

\begin{proposition}(Van der Corput's lemma)
    \label{Prop_vanderCorput}
    Let $J = (a,b) \subseteq \R$ be an interval and $f: J \rightarrow \R$ be a convex or concave function which is continuously differentiable on $J$ and assume that $ | f'(x) | > 0$ for all $x \in J$. Moreover, let $\phi : J \rightarrow \R $ be a $\mathcal{C}^1(J)$ function. Then, there exists an universal constant $c \in (0, \infty)$ such that for all $\lambda \in \R \backslash \{0\}$ , it holds
    \begin{equation*}
        \left| \int_{a}^b \phi(x) e^{i \lambda f(x) } dx \right| \leq c \left( \max_{x \in J} | \phi(x) | + \int_{J} | \phi'(x) | dx \right) \frac{1}{| \lambda |}.
    \end{equation*}
\end{proposition}

\begin{proposition}
    \label{Prop_Aux_1}
    Let $V,W:\R \rightarrow [0, \infty)$ be two Orlicz functions or let $V(x) = |x|^p$ and $W(x) = |x|^q$ with $p, q \in (0,1)$ and $p \neq q$. For $\lambda \in [-1,1]$, let $\mathcal{M}_{\lambda}$ denote the set of points in $\R$ where the function 
    \begin{equation*}
        x \mapsto V(x) + \lambda W(x)
    \end{equation*}
    changes its curvature. Then $\overline{\mathcal{M}_{\lambda}}$ cannot contain an interval.
\end{proposition}
\begin{proof}
    The proof is done by contradiction. We assume there exists a $\bar{\lambda} \in \R$ and a non-trivial interval $I \subseteq \R$ such that the set of inflection points of $x \mapsto V(x) + \bar{\lambda} W(x)$ is dense in $I$. Then, by our definition of an Orlicz function, we find a non-trivial subinterval $J \subseteq I$, where both $V$ and $W$ are $\mathcal{C}^2$. On $J$, the set of inflection points of $V+ \bar{\lambda} W$ is still dense and thus $V'' + \bar{\lambda} W''$ is not continuous on $J$, which is an immediate contradiction.  
\end{proof}

\begin{proposition}
    \label{Prop_Aux_3}
    Let $V,W:\R \rightarrow [0, \infty)$ be two Orlicz functions or let $V(x)=|x|^p$ and $W(x)=|x|^q$ with $p,q \in (0,1)$ and $p \neq q$. Then, there exists a compact interval $I \subseteq \R$, some $\delta \in (0, \infty)$ and a family of intervals $(J_{\lambda})_{\lambda \in [-1,1]}$ with $J_{\lambda} \subseteq I$ for all $\lambda \in [-1,1]$ and $\inf_{\lambda \in [-1,1]} | J_{\lambda}| \geq \delta $ such that, for $\lambda \in [-1,1]$, the function 
    \begin{equation*}
        x \mapsto V(x) + \lambda W(x)
    \end{equation*}
    is $\mathcal{C}^2$, strictly monotone and does not change its curvature on $J_{\lambda}$.
\end{proposition}
\begin{proof}
First we show that there exists a compact interval $I \subseteq \R$ and some $\delta' \in (0, \infty)$ such that, for all $\lambda \in [-1,1]$, there exists an interval $J_{\lambda}' \subseteq I$ with $x \mapsto V(x) + \lambda W(x)$ being strictly monotone throughout $J_{\lambda}'$ and $|J_{\lambda}'| \geq \delta'$. 
We claim that there exists an interval $I \subseteq \R$ such that for all $\lambda \in [-1,1]$, the function $V + \lambda W$ is not constant on $I$. If such an interval $I$ would not exist, for every $M \in (0,\infty)$, we would find a $\lambda \in [-1,1]$ such that $ V + \lambda W$ is constant (and hence $0$) on $[-M, M]$, which is a contradiction to the linear independence of $V$ and $W$. Since, for every $\lambda \in [-1,1] $, the function $V+ \lambda W$ is not constant on $I$, we can define $ J_{\lambda}' \subseteq I$ to be the largest subinterval of $I$ such that $V+ \lambda W$ is either strictly monotone increasing on $J_{\lambda}'$ or strictly monotone decreasing on $J_{\lambda}'$. Moreover, we claim that there exists a $\delta' \in (0, \infty)$ such that $\inf_{\lambda \in [-1,1]} | J_{\lambda}' | \geq \delta'$. Assume that such a positive $\delta'$ with $\inf_{\lambda \in [-1,1]} | J_{\lambda}' | \geq \delta'$ does not exist. Then, for every $n \in \N$, we find a $\lambda_n \in [-1,1]$ such that $|J_{\lambda_n}'| < \frac{1}{n}$. Since $(\lambda_n)_{n \in \N} $ is a bounded sequence, there exists a convergent subsequence $(\lambda_{n_k})_{k \in \N}$ with limit $\bar{\lambda} \in [-1,1]$. By construction, $V+ \bar{\lambda} W$ is constant on $I$ which is an immediate contradiction. Thus, there exists a $\delta' \in (0, \infty)$ with $\inf_{\lambda \in [-1,1]} | J_{\lambda}'| \geq \delta'$.  
\par{}
Now, for all $\lambda \in [-1,1]$, let $J_{\lambda} \subseteq J_{\lambda}'$ be the largest subinterval of $J_{\lambda}$ such that $x \mapsto V(x) + \lambda W(x)$ does not change its curvature throughout $J_{\lambda}$. By Proposition \ref{Prop_Aux_1} it follows that there must exist a $\delta \in (0, \delta')$ such that $\inf_{\lambda \in [-1,1]} | J_{\lambda} | \geq \delta$.
\par{}
The case where $V(x)=|x|^p$ and $W(x)=|x|^q$ can be argued similarly. 
\end{proof}

\begin{proposition}
    \label{Prop_Main_Tool}
    Let $\left( (t_n,s_n) \right)_{n \in \N} \subseteq \R^2$ be a sequence with $|t_n| \geq |s_n|$ for all $n \in \N$ and $\lim_{n \rightarrow \infty} |t_n| = \infty$. Further, let $V,W: \R \rightarrow [0, \infty)$ be two linearly independent Orlicz functions or $V(x)=|x|^p$ and $W(x)=|x|^q$ with $p, q \in (0,1)$ and $p \neq q$. Then, there exists an $\varepsilon \in (0,1)$ and some sequence $(\delta_n)_{n \in \N}$ with $\delta_n \rightarrow 0$ such that
    \begin{equation*}
     \left | \frac{1}{Z_{\alpha, \beta}}\int_{\R}  e^{\alpha V(x) + \beta W(x) } e^{i t_n V(x) + i s_n W(x)} \right | \leq 1 - \varepsilon + \delta_n, \quad n \in \N. 
    \end{equation*}
\end{proposition}
\begin{proof}
    We prove this proposition by applying Van der Corput's lemma (see Proposition \ref{Prop_vanderCorput}). For $n \in \N$, define $\lambda_n := \frac{s_n}{t_n}$ (we can assume that $t_n \neq 0$ for all $n \in \N$, since $\lim_{n \rightarrow \infty} |t_n| = \infty$). Then $\lambda_n \in [-1,1]$ for all $n \in \N$ by assumption. By Proposition \ref{Prop_Aux_3}, there exists a $\delta \in (0,\infty)$, a compact interval $I \subseteq \R$ and a sequence of intervals $(J_{\lambda_n})_{n \in \N}$ with $\inf_{n \in \N} | J_{\lambda_n}| \geq \delta $ and $J_{\lambda_n} \subseteq I$ for all $n \in \N$ such that the function
    \begin{equation*}
        x \mapsto V(x) + \lambda_n W(x)
    \end{equation*}
    is $\mathcal{C}^2$, strictly monotone and does not change its curvature on $J_{\lambda_n}$. By Proposition \ref{Prop_vanderCorput}, we thus get
    \begin{equation*}
        \left |  \frac{1}{Z_{\alpha, \beta}}\int_{J_{\lambda_n}} e^{\alpha V(x) + \beta W(x)} e^{ i t_n \left( V(x) + \lambda_n W(x) \right)} dx  \right | \leq \frac{c_n}{|t_n|},
    \end{equation*}
    where we have
    \begin{align*}
        c_n &= c \left( \max_{x \in J_{\lambda_n} } \left[  e^{\alpha V(x) + \beta W(x)} \right] + \int_{J_{\lambda_n}}\left | \alpha V'(x) + \beta W'(x)  \right | e^{\alpha V(x) + \beta W(x)} dx \right).
    \end{align*}
    The construction of $J_{\lambda_n}$ implies that all $c_n$ are uniformly bounded by some absolute constant $C \in (0, \infty)$. We thus have
    \begin{equation}
    \label{Eq_Estimate_1}
        \left |  \frac{1}{Z_{\alpha, \beta}}\int_{J_{\lambda_n}} e^{\alpha V(x) + \beta W(x)} e^{ i t_n \left( V(x) + \lambda_n W(x) \right)} dx  \right | \leq \frac{C}{|t_n|} =: \delta_n.
    \end{equation}
    Since all intervals $J_{\lambda_n}$ are contained in a fixed compact interval $I \subseteq \R$ and all $J_{\lambda_n}$ have length of at least $\delta \in (0, \infty)$, there must exist an $\varepsilon \in (0,1)$ such that 
    \begin{equation}
    \label{Eq_Estimate_2}
     \left |  \frac{1}{Z_{\alpha, \beta}}\int_{J_{\lambda_n}^c } e^{\alpha V(x) + \beta W(x)} e^{ i t_n \left( V(x) + \lambda_n W(x) \right)} dx  \right | \leq   1- \varepsilon.
    \end{equation}
    The estimates in \eqref{Eq_Estimate_1} and \eqref{Eq_Estimate_2} now show the claim. 
\end{proof}

\begin{lem}
    \label{Lem_AuxL3}
   Let $V, W : \R \rightarrow \R $ be two linearly independent Orlicz functions or $V(x) = |x|^p$ and $W(x) = |x|^q$ with $p, q \in (0,1)$ and $p \neq q$. Let $ (\alpha , \beta) \in \R^2$ such that $ \mu_{\alpha, \beta} $ exists and let $\left( X_i \right)_{i \in \N}$ be an iid sequence with $X_1 \sim \mu_{\alpha, \beta}$. We assume that there exists a $k \in \N_{\geq 2}$ with
   \begin{equation}
   \label{Eq_kth_moment}
    \mathbb{E}_{\mu_{\alpha, \beta}} \left[ \left | \left | \left(V(X),W(X) \right) \right | \right |_2^k \right] < \infty.  
   \end{equation}
   We set $(R,t) := \left( \mathbb{E} \left[ V(X_1) \right] , \mathbb{E} \left[ W(X_1) \right] \right)$ and, for $n \in \N$, we denote the distribution function of the vector 
\begin{equation*}
\frac{1}{\sqrt{n} } \sum_{i=1}^n \left( V(X_i) - R , W(X_i) - t\right)   
\end{equation*}
as $F_n$. Then, for any $(x,y) \in \R^2$, we can write
\begin{equation*}
F_n(x,y) = G^{ \Sigma} (x,y) + \sum_{j=1}^k n^{-j/2} F_j(x,y) + R_n(x,y),
\end{equation*}
where the remainder $R_n$ satisfies $R_n(x,y)=o(n^{-k/2})$, uniformly in $(x,y) \in \R^2$. $G^{\Sigma}$ is the distribution function of a normal distribution with mean $0 \in \R^2$ and covariance matrix $\Sigma \in \R^{2 \times 2}$ given by
\begin{equation*}
    \Sigma = \Cov \left[ \left(V(X_1), W(X_1) \right) \right].
\end{equation*}
Moreover, the functions $F_j$, for $j = 1, \ldots, k$, can be computed via
\begin{equation*}
F_j(x,y) = \int_{- \infty}^x \int_{-\infty}^y g^{\Sigma}(t,s) \pi_j(t,s) dt ds, \quad (x,y) \in \R^2,
\end{equation*}
where $g^{\Sigma}$ is the density of $G^{\Sigma}$ and $\pi_j$ is a polynomial in two variables of degree at most $3 j$. 
\end{lem}

\begin{proof}
Since $V,W$ are linearly independent, the covariance matrix $\Sigma := \Cov \left(V(X_1), W(X_1) \right)\in \R^{2 \times 2}$ (which exists by \eqref{Eq_kth_moment}) is positive definite. Thus, the result is a consequence of the elaborations in \cite[Chapter 2]{bhattacharya2010}, provided the assumption
\begin{equation*}
% \label{Eq_CharacFunc}
     \limsup_{||(t,s)||_2 \rightarrow \infty} \left | \mathbb{E}_{\mu_{\alpha, \beta }} \left[ e^{i t V(X) + is W(X)}\right] \right | < 1
 \end{equation*}
is satisfied. To that end, let $ \left( (t_n, s_n) \right)_{n \in \N}$ be a sequence of real numbers with $ || (t_n, s_n) ||_2 \rightarrow \infty$ as $n \rightarrow \infty$, where we need to show that
\begin{equation}
\label{Eq_discrete_Char_fct}
  \limsup_{n \rightarrow \infty} \left | \mathbb{E}_{\mu_{\alpha, \beta }} \left[ e^{i t_n V(X) + is_n W(X)}\right] \right | < 1 .  
\end{equation}

We define two sequences of natural numbers $(n_k)_{k \in \N}$ and $(m_k)_{k \in \N}$ recursively (we set $n_0:=m_0:=0$)
\begin{equation*}
n_k:= \min \left \{ n \in \N : n > n_{k-1} , |s_n| \leq |t_n|  \right\} \qquad m_k:= \min \left \{ n \in \N : n > m_{k-1} , |s_n| > |t_n|  \right\}, \quad \text{for } k \in \N.
\end{equation*}
Then, by construction, $(n_k)_{k \in \N}$ and $(m_k)_{k \in \N}$ form a partition of $\N$. In the following, we only consider the case where both those sequences tend to infinity. The limit in \eqref{Eq_discrete_Char_fct} can then be written as
\begin{align*}
  \limsup_{n \rightarrow \infty} \left | \mathbb{E}_{\mu_{\alpha, \beta }} \left[ e^{i t_n V(X) + is_n W(X)}\right] \right | &= \limsup_{k \rightarrow \infty} \max \left \{ \left | \mathbb{E}_{\mu_{\alpha, \beta }} \left[ e^{i t_{n_k} V(X) + is_{n_k} W(X)}\right] \right | , \left | \mathbb{E}_{\mu_{\alpha, \beta }} \left[ e^{i t_{m_k} V(X) + is_{m_k} W(X)}\right] \right |\right\}.  \\
\end{align*}
By Proposition \ref{Prop_Aux_3}, we find $\varepsilon, \varepsilon' \in (0, 1)$ and two sequences $(\delta_k)_{k \in \N}$ and $(\delta'_k)_{k \in \N}$ with $\delta_k \rightarrow 0, \delta_k' \rightarrow 0$ such that
\begin{align*}
 \limsup_{k \rightarrow \infty} \max \left \{ \left | \mathbb{E}_{\mu_{\alpha, \beta }} \left[ e^{i t_{n_k} V(X) + is_{n_k} W(X)}\right] \right | , \left | \mathbb{E}_{\mu_{\alpha, \beta }} \left[ e^{i t_{m_k} V(X) + is_{m_k} W(X)}\right] \right |\right\} &=    
 \limsup_{k \rightarrow \infty} \max \left \{ 1- \varepsilon + \delta_k , 1- \varepsilon' + \delta_k' \right\} \\
 & = 1 - \min \{\varepsilon, \varepsilon' \} < 1.
\end{align*}
This shows the desired result. 

\end{proof}

% % % % % % % % % % % % % % % % % % % % % % % % %
\subsection{Proofs of the Theorems}
% % % % % % % % % % % % % % % % % % % % % % % % %

We start by proving a general lemma needed in several of the upcoming proofs.

\begin{lem}
    \label{Lem_Asympt_Int}
    Let the assumptions of Lemma \ref{Lem_AuxL3} be satisfied for a pair $(\alpha, \beta) \in (- \infty, 0) \times  (0,\infty)$. Then, for the sequence of distribution functions $(F_n)_{n \in \N}$, we have
    \begin{equation*}
    \int_{x=0}^{\infty} \int_{y=-\infty}^0 e^{ \sqrt{n} \alpha x + \sqrt{n} \beta y} d F_n(x,y) = \int_{x=0}^{\infty} \int_{y=-\infty}^0 e^{ \sqrt{n} \alpha x + \sqrt{n} \beta y} d G^{\Sigma}(x,y) + o(n^{-1}).
    \end{equation*}
If $\beta=0$ in the integral above and $t \in \R$, then we get
\begin{equation*}
 \int_{x=0}^{\infty} \int_{y=-\infty}^t e^{ \sqrt{n} \alpha x + \sqrt{n} \beta y} d F_n(x,y) = \int_{x=0}^{\infty} \int_{y=-\infty}^t e^{ \sqrt{n} \alpha x} d G^{\Sigma}(x,y) + O(n^{-1}),   
\end{equation*}
where the implicit constant in the $O$-term is independent of $t$.
\end{lem}
\begin{proof}
We start with the first case, where $(\alpha, \beta) \in (- \infty, 0) \times (0, \infty)$.
Since $\mu_{\alpha, \beta}$ exists for a pair $(\alpha, \beta) \in (- \infty, 0) \times (0, \infty)$, we have that
\begin{equation*}
  \mathbb{E}_{\mu_{\alpha, \beta}} \left[ \left | \left | \left(V(X),W(X) \right) \right | \right |_2^2 \right] < \infty,   
\end{equation*}
and thus we can write 
\begin{equation*}
    F_n(x,y)= G^{\Sigma}(x,y) + \sum_{j=1}^2 n^{-j/2}F_j(x,y) + R_n(x,y),
\end{equation*}
with $R_n(x,y)=o(n^{-1})$, uniformly in $(x,y) \in \R^2$ and the density of $F_j$ is of the form 
$g^{\Sigma}(x,y) \pi_j(x,y)$ with $\pi_j$ having degree $3j$.
This leads to (recall that $\alpha \in (- \infty,0)$)
\begin{align*}
    \int_{x=0}^{\infty} \int_{y=- \infty}^0 e^{ \sqrt{n} \alpha x + \sqrt{n} \beta y} d R_n(x,y) 
    & \leq R_n \left((-\infty,0) \times (0, \infty) \right) \\
    & = o(n^{-1}).
\end{align*}
There exist constants $c_1, c_2 \in (0, \infty)$ such that, for $j=1,2$, we have
\begin{equation*}
   \left | g^{\Sigma}(x,y) \pi_j(x,y) \right | \leq c_1 e^{- c_2( x^2 +y^2)}, \quad \text{for all }(x,y) \in \R^2. 
\end{equation*}
From that we infer (we denote $\gamma := \min \{|\alpha|, \beta \}$)
\begin{align*}
    \int_{x=0}^{\infty} \int_{y=- \infty}^{0} e^{ \sqrt{n} \alpha x + \sqrt{n} \beta y} d F_j(x,y)
    & \leq  \int_{x=0}^{\infty} \int_{y=- \infty}^{0} e^{ \sqrt{n} \alpha x + \sqrt{n} \beta y} c_1 e^{- c_2( x^2 +y^2)} d y dx  \\
    & \leq c_1 \left( \int_{0}^{\infty} e^{- \sqrt{n} \gamma x } e^{- c_2 x^2} dx  \right)^2 \\
    & = O \left( n^{-1} \right).
\end{align*}
This implies that
\begin{equation*}
    \sum_{j=1}^2 \int_{x=0}^{\infty} \int_{y=-\infty}^{0} e^{ \sqrt{n} \alpha x + \sqrt{n} \beta y} d \left( n^{-j/2} F_j(x,y) \right) = o(n^{-1}).
\end{equation*}
This establishes the first claim. Now let $ \alpha \in (- \infty, 0)$ and $\beta=0$. Then, again, we have
\begin{equation*}
   \mathbb{E}_{\mu_{\alpha, 0}} \left[ \left | \left | \left(V(X),W(X) \right) \right | \right |_2^2 \right] < \infty, 
\end{equation*}
allowing us to write
\begin{equation*}
   F_n(x,y)= G^{\Sigma}(x,y) +  n^{-1/2}F_1(x,y) + R_n(x,y),
\end{equation*}
with $R_n(x,y)=O(n^{-1})$, uniformly in $(x,y) \in \R^2$ and the density of $F_j$ is of the form 
$g^{\Sigma}(x,y) \pi_1(x,y)$ with $\pi_1$ having degree $3$ and $g^{\Sigma}$ is the density of a bivariate normal distribution with mean $0$ and covariance matrix $\Sigma \in \R^{2 \times 2}$. As before, we get
\begin{align*}
    \int_{x=0}^{\infty}\int_{y=- \infty}^t e^{\alpha \sqrt{n}x} dF_n(x,y) 
    &= \int_{x=0}^{\infty}\int_{y=- \infty}^t e^{\alpha \sqrt{n}x} dG^{\Sigma}(x,y) + \int_{x=0}^{\infty}\int_{y=- \infty}^t e^{\alpha \sqrt{n}x} d \left( n^{-1/2} F_1(x,y) \right) + O(n^{-1}).
\end{align*}
We further find constants $c_1, c_2 \in (0, \infty)$ with
\begin{equation*}
    \left | \pi_1(x,y) g^{\Sigma} (x,y) \right | \leq c_1 e^{- c_2 ( x^2 + y^2)}.
\end{equation*}
This implies
\begin{align*}
  \int_{x=0}^{\infty}\int_{y=- \infty}^t e^{\alpha \sqrt{n}x} d \left( n^{-1/2} F_1(x,y) \right) & \ll n^{-1/2} \int_{x=0}^{\infty} e^{\alpha \sqrt{n}x} e^{-c_2 x^2} \underbrace{\int_{y=- \infty}^t e^{-c_2 y^2} d y}_{\ll 1} dx \\
  & \ll n^{-1/2} \int_{x=0}^{\infty} e^{\alpha \sqrt{n}x} dx \\
  & \ll n^{-1}.
\end{align*}
\end{proof}

The next lemma gives an estimate for the asymptotic volume of a high dimensional Orlicz ball. Such a result was first obtained in \cite[Theorem A]{KP2021}, where the error term was of order $o(1)$. For the proof of Theorem \ref{Thm_SLLN_CLT}, we need an error of order $O(n^{-1/2})$.
\begin{lem}
    \label{Lem_Vol_OB}
    Let $R \in (0, \infty)$ and $V: \R \rightarrow [0, \infty)$ be an Orlicz function. Then the asymptotic volume of $\mathbb{B}^{(n,V)}_R$ is given by
    \begin{equation*}
        \vol_{n} \left( \mathbb{B}_{R}^{(n,V)} \right) = \frac{e^{- \alpha_* n R } Z_{\alpha_*, 0}^n}{\sqrt{2 \pi n} \sigma_* | \alpha_* |} \left( 1 + O\left(n^{-1/2} \right) \right).
    \end{equation*}
\end{lem}
\begin{rem}
We note that Lemma \ref{Lem_Vol_OB} is a special case of \cite[Theorem 1.1]{barthe2023volume}. Nevertheless, in the following, we present an alternative proof, resting on the methods developed in this article.
\end{rem}
\begin{proof}(Lemma \ref{Lem_Vol_OB})
We have
\begin{align*}
 \vol_{n} \left( \mathbb{B}_{R}^{(n,V)} \right) &=  \int_{\R^n}\mathbbm{1}_{ \mathbb{B}^{(n,V)}_R }(x) dx_1 \ldots dx_n \\
    &= Z_{\alpha_*, 0}^{n} \int_{\R^n}
   \mathbbm{1}_{ \left [ \sum_{i=1}^n V(x_i) \leq nR \right ]}  \frac{1}{Z_{\alpha_*, 0}^n} e^{ (\alpha_* - \alpha_*) \sum_{i=1}^n V(x_i)  } dx_1 \ldots dx_n \\ 
   &= Z_{\alpha_*, 0}^{n}e^{- n \alpha_* R } \int_{x=- \infty}^0 \int_{y =- \infty }^{\infty}  e^{  - \alpha \sqrt{n} x  } dF_n(x,y) , 
\end{align*}
where $F_n$ is the distribution function of the vector
\begin{equation*}
    \frac{1}{\sqrt{n}} \sum_{i=1}^n \left( V(X_i) - R, W(X_i) - m \right),
\end{equation*}
with $(X_i)_{i \in \N}$ being an iid sequence all distributed according to $\mu_{\alpha_*, 0}$. By Lemma \ref{Lem_AuxL3} together with Lemma \ref{Lem_Asympt_Int}, we get
\begin{align*}
    \int_{x=- \infty}^0 \int_{y =- \infty }^{\infty}  e^{  - \alpha \sqrt{n} x  } dF_n(x,y) &= \int_{x=- \infty}^0 \int_{y =- \infty }^{\infty}  e^{  - \alpha \sqrt{n} x  } d G^{\Sigma}(x,y) + O(n^{-1}) \\
    &= \frac{1}{\sqrt{2 \pi } \sigma_*} \int_{x=0}^{\infty} e^{\alpha_* \sqrt{n} x} e^{-\frac{x^2}{2 \sigma_*}}dx +O(n^{-1}) \\
    &= \frac{1}{\sqrt{2 \pi n }\sigma_*} \int_{x=0}^{\infty} e^{\alpha_*  x} e^{-\frac{x^2}{2 n\sigma_*}}dx +O(n^{-1}) \\
    &= \frac{1}{\sqrt{2 \pi n }\sigma_*} \int_{x=0}^{\infty} e^{\alpha_*  x} dx + \frac{1}{\sqrt{2 \pi n }\sigma_*} \int_{x=0}^{\infty} e^{\alpha_*  x} \left( e^{-\frac{x^2}{2 n\sigma_*}}-1 \right) dx + O(n^{-1}) \\
    &= \frac{1}{\sqrt{2 \pi n }\sigma_* | \alpha_*|} + \frac{1}{\sqrt{2 \pi n }\sigma_*} \int_{x=0}^{\infty} e^{\alpha_*  x} \left( e^{-\frac{x^2}{2 n\sigma_*}}-1 \right) dx + O(n^{-1}).
\end{align*}
We take a closer look at the second integral in the expression above, where we get (we use the estimate $|e^t -1| \leq |t| \max \{e^t, 1 \}$ for $t \in \R$)
\begin{align*}
 \frac{1}{\sqrt{2 \pi n }\sigma_*} \left | \int_{x=0}^{\infty} e^{\alpha_*  x} \left( e^{-\frac{x^2}{2 n\sigma_*}}-1 \right) dx \right | 
  & \ll \frac{1}{n^{3/2}}     
\int_{x=0}^{\infty} x^2 e^{\alpha_*  x} \underbrace{\max \left \{ 1 , e^{-\frac{x^2}{2n \sigma_*}} \right \}}_{\leq 1} dx \\
& \ll \frac{1}{n^{3/2}}.
\end{align*}
This finishes the proof.
\end{proof}

\begin{proof}(Theorem \ref{ThmSLD})
We only prove $(i)$, since the proof of $(ii)$ can be carried out in a similar way. By Lemma \ref{Lem_AuxL2}, there exists an $\varepsilon \in (0, \infty)$ such that, for all $t \in (m, m+ \varepsilon)$, we find a pair $(\alpha, \beta) \in (- \infty, 0) \times (0, \infty)$ with
\begin{equation*}
    \left( \begin{array}{c}
         R  \\
         t 
    \end{array} \right)
     = \left( \begin{array}{c}
         \mathbb{E}_{\mu_{\alpha, \beta}} \left[ V(X) \right]  \\
          \mathbb{E}_{\mu_{\alpha, \beta}} \left[ W(X) \right] 
    \end{array} \right).
\end{equation*}
Let $(X_i)_{i \in \N}$ be an iid sequnce of random variables with $X_1 \sim \mu_{\alpha, \beta}$. Then we can introduce the following sequence of centered random variables
\begin{equation*}
    \frac{1}{\sqrt{n}} \sum_{i=1}^n \left( Y_i, Z_i \right) := \frac{1}{\sqrt{n}} \sum_{i=1}^n  \left( V(X_i) - R, W(X_i) -t\right), \quad n \in \N.  
\end{equation*}
We denote the distribution function of this sequence by $F_n$ and recall that $F_n $ has an Edgeworth expansion by Lemma \ref{Lem_AuxL3}. We now get
\begin{align*}
    \vol_{n} \left( \mathbb{B}^{(n,V)}_R \cap \left(\mathbb{B}^{(n,W)}_t \right)^c\right) 
    &= \int_{\R^n}\mathbbm{1}_{ \mathbb{B}^{(n,V)}_R }(x) \mathbbm{1}_{  \left( \mathbb{B}^{(n,W)}_t \right)^c }(x) dx_1 \ldots dx_n \\
    &= Z_{\alpha, \beta}^{n} \int_{\R^n}
   \mathbbm{1}_{ \left [ \sum_{i=1}^n V(x_i) \leq nR \right ]} \mathbbm{1}_{  \left [ \sum_{i=1}^n W(x_i) > n t \right ] }  \frac{1}{Z_{\alpha, \beta}^n} e^{ ( \alpha - \alpha) \sum_{i=1}^n V(x_i) + (\beta- \beta)  \sum_{i=1}^n W(x_i)  } dx_1 \ldots dx_n \\
   &=  Z_{\alpha, \beta}^{n} e^{-n \alpha R - n \beta t} \mathbb{E} 
   \left[ \mathbb{1}_{ \left [ \frac{1}{\sqrt{n}} \sum_{i=1}^n  Y_i \leq 0 \right ] } \mathbb{1}_{ \left[   \frac{1}{\sqrt{n}} \sum_{i=1}^n  Z_i > 0 \right] }  e^{ \sqrt{n} \alpha \frac{1}{\sqrt{n}} \sum_{i=1}^n  (Y_i-R) + \sqrt{n} \beta \frac{1}{\sqrt{n}} \sum_{i=1}^n  (Z_i-t) } \right]\\
    &= Z_{\alpha, \beta}^n e^{-n \alpha R - n \beta t} \int_{x=0}^{\infty} \int_{y= - \infty}^0 e^{- \sqrt{n} \beta x - \sqrt{n} \alpha y} dF_n(x,y).
\end{align*}
By Lemma \ref{Lem_AuxL3}, we can write
\begin{equation*}
    F_n(x,y) = G^{\Sigma}(x,y) + \sum_{j=1}^{2} n^{-j/2} F_j(x,y) + R_n(x,y),
\end{equation*}
where $G^{\Sigma}$ is the distribution function of a normal distribution with mean $0$ and variance matrix $\Sigma \in \R^{2 \times 2}$ and $R_n(x,y) = o(n^{-1})$, uniformly in $(x,y) \in \R^2$. By Lemma \ref{Lem_Asympt_Int} we get
\begin{align*}
\int_{x=0}^{\infty} \int_{y= - \infty}^0 e^{- \sqrt{n} \beta x - \sqrt{n} \alpha y} dF_n(x,y) & = 
\int_{x=0}^{\infty} \int_{y= 0}^{\infty } e^{- \sqrt{n} \beta x + \sqrt{n} \alpha y} dG^{\Sigma}(x,y) + o(n^{-1}) \\
& = \int_{x=0}^{\infty} \int_{y= 0}^{\infty } e^{- \sqrt{n} ( \beta x - \alpha y)} \frac{1}{2 \pi | \Sigma |^{1/2}} e^{ - \frac{1}{2} (x,y) \Sigma^{-1}
    \left(
    \begin{smallmatrix}
    x\\y
    \end{smallmatrix}
    \right)
    }dx dy + o(n^{-1}) \\
    &=
    \frac{1}{2 \pi n | \Sigma |^{1/2}}\int_{x=0}^{\infty}\int_{y=0}^{\infty}e^{-\beta x + \alpha y} \underbrace{ e^{ - \frac{1}{2n} (x,y) \Sigma^{-1}
    \left(
    \begin{smallmatrix}
    x\\y
    \end{smallmatrix}
    \right)}}_{=:f_n(x,y)} dx dy+o(n^{-1}).
\end{align*}
We now observe that $f_n(x,y) \leq 1$ for all $(x,y) \in \R^2$ and for all $n \in \N$. Since $e^{- \beta x + \alpha y}$ is integrable, the dominated convergence theorem yields
\begin{align*}
   \frac{1}{2 \pi n | \Sigma |^{1/2}}\int_{x=0}^{\infty}\int_{y=0}^{\infty}e^{-\beta x + \alpha y} f_n(x,y) dx dy 
   &=\frac{1}{2 \pi n | \Sigma |^{1/2}}\int_{x=0}^{\infty}\int_{y=0}^{\infty}e^{-\beta x + \alpha y}  dx dy + o(n^{-1}) \\
   &=\frac{1}{2 \pi n (- \alpha) \beta | \Sigma |^{1/2}} +o(n^{-1}).
\end{align*}

By Lemma \ref{Lem_Vol_OB}, it holds that
\begin{equation*}
    \vol_{n} \left( \mathbb{B}^{(n,V)}_R \right) = e^{ -\alpha^{*} n R }Z_{\alpha^{*},0}^{n}\frac{1}{\sqrt{2 \pi n} \sigma_{*} | \alpha_* |} \left( 1 + O \left(n^{-1/2} \right) \right), 
\end{equation*}
where $\sigma_{*}^2 = \mathbb{V}_{\mu_{\alpha_*, 0}} \left[ V(X) \right]$. This implies that
\begin{equation*}
    \frac{\vol_{n} \left( \mathbb{B}^{(n,V)}_R \cap \mathbb{B}^{(n,W)}_t \right)}{\vol_{n} \left( \mathbb{B}^{(n,V)}_R \right)} = 1 - e^{ -n \left(\varphi(\alpha^{*},0)- \varphi( \alpha, \beta) + (\alpha - \alpha^{*})R + \beta t \right)} \frac{| \alpha_*|}{\sqrt{2 \pi n} | \alpha | \beta \bar{\sigma}} \left( 1 + o(1) \right),
\end{equation*}
where
\begin{equation*}
\bar{\sigma}^2 = \frac{\sigma_*^2}{| \Sigma |} 
\end{equation*}
is the second entry in the diagonal of $\Sigma^{-1}$. Since, by construction,
\begin{equation*}
\left( \begin{array}{c}
         \alpha  \\
         \beta 
    \end{array} \right)
    =
\left(\nabla \varphi \right)^{-1} \left( \begin{array}{c}
         R  \\
         t 
    \end{array} \right) 
\quad \text{and} \quad
    \left( \begin{array}{c}
         \alpha^{*}  \\
         0 
    \end{array} \right) 
    =
\left(\nabla \varphi \right)^{-1} \left( \begin{array}{c}
         R  \\
         m 
    \end{array} \right),
\end{equation*}
we recover the definition of $\I$ and thus the estimate in \eqref{Eq_SharpDev_V_W}.

%This finishes the case where $t < W(V^{-1}(R))$. 
% It is an easy observation that, for all $n \in \N$, 
% \begin{equation*}
% \sup_{x \in \mathbb{B}^{(n,V)}_R} \frac{1}{n} \sum_{i=1}^n W(x_i) = W(V^{-1}(R)),
% \end{equation*}
% and thus, if $t \geq W(V^{-1}(R))$, we have
% \begin{equation*}
%      \frac{\vol^{(n)} \left( \mathbb{B}^{(n,V)}_R \cap \mathbb{B}^{(n,W)}_t \right)}{\vol^{(n)} \left( \mathbb{B}^{(n,V)}_R \right)} = 1.
% \end{equation*}
\end{proof}

\begin{proof}(Theorem \ref{ThmSLD_thinshell})
We prove a more general result, where instead of the $\ell_2$-norm we work with another Orlicz function $W$, such that $V$ and $W$ satisfy Assumption \ref{Ass_A}. 
We define $\varepsilon \in (0, \infty)$ to be the largest constant such that, for any $ \delta \in (0, \varepsilon)$, we find pairs $ (\alpha_1, \beta_1) \in (- \infty, 0) \times (0, \infty)$ and $(\alpha_2, \beta_2) \in (- \infty, 0)^2$ with
\begin{equation}
\label{Eq_Sol_Expectation}
    \nabla \varphi(\alpha_1, \beta_1) = \left(\begin{array}{c}
R   \\ m + \delta
    \end{array} \right)
    \quad \text{and} \quad 
      \nabla \varphi(\alpha_2, \beta_2) = \left(\begin{array}{c}
R   \\ m - \delta
    \end{array} \right).
\end{equation}
This is possible by Lemma \ref{Lem_AuxL2} and Lemma \ref{Lem_Aux4}. We now obtain
\begin{align*}
\mathbb{P} \left[ \left | \frac{1}{n} \sum_{i=1}^n W \left( X_i^{(n,V)} \right) - m \right| \geq \delta   \right] &=
    \mathbb{P} \left[  \frac{1}{n} \sum_{i=1}^n  W \left( X_i^{(n,V)} \right) \geq m + \delta   \right] +  \mathbb{P} \left[  \frac{1}{n} \sum_{i=1}^n  W \left( X_i^{(n,V)} \right) \leq m - \delta   \right].
\end{align*}
Analogously to the proof of Theorem \ref{ThmSLD}, we now get
\[
\mathbb{P} \left[ \left | \frac{1}{n} \sum_{i=1}^n W \left( X_i^{(n,V)} \right) - m \right| \geq \delta   \right] = \frac{1}{\sqrt{2 \pi n} \sigma_1}e^{-n \mathbb{I}(R,m+ \delta )} (1 + o(1))+   \frac{1}{\sqrt{2 \pi n} \sigma_2}e^{-n \mathbb{I}(R,m- \delta )} (1 + o(1)),
\]
where $\mathbb{I}$ is the same GRF as in Theorem \ref{ThmSLD} and
\[
\sigma_i^2:=\mathbb{V}_{\mu_{\alpha_i, \beta_i}}\left[ V(X)\right], \quad i=1,2.
\]
The claim now follows by defining
\[
\tilde{\sigma}:= \begin{cases}
    \sigma_1, & \text{if} \quad \mathbb{I}(R, m+ \delta) > \mathbb{I}(R, m - \delta) \\
    \sigma_2, & \text{if} \quad \mathbb{I}(R, m+ \delta) < \mathbb{I}(R, m - \delta) \\
    \frac{\sigma_1 \sigma_2}{\sigma_1+ \sigma_2}, & \text{if} \quad \mathbb{I}(R, m+ \delta) = \mathbb{I}(R, m - \delta).
\end{cases}
\]
\end{proof}

\begin{proof}(Theorem \ref{Thm_SLLN_CLT})
We start with the proof of the Berry--Esseen result, where we recall that there exists an $\alpha_* \in (- \infty, 0)$ such that
\begin{equation*}
  \left( \begin{array}{c}
         \mathbb{E}_{\mu_{\alpha_*, 0}}[V(X)]   \\
         \mathbb{E}_{\mu_{\alpha_*, 0}}[W(X)] 
    \end{array} \right)
    = \left( \begin{array}{c}
        R   \\
         m
    \end{array} \right).
\end{equation*}
This leads to
\begin{align*}
\mathbb{P} \left[ \frac{1}{\sqrt{n}}  \sum_{i=1}^n \left [ W \left( X_i^{(n,V)} \right) - m \right ] \leq t \right] &= \frac{1}{\vol_{n}  \left( \mathbb{B}_R^{(n,V)} \right) } \int_{\R^n } \mathbb{1}_{  \mathbb{B}_R^{(n,V)}}(x) \mathbb{1}_{ \left[\frac{1}{ \sqrt{n}}\sum_{i=1}^n \left( W(x_i) - m \right) \leq t \right]} dx_1 \ldots d x_n \\
&= \frac{Z_{\alpha_*}^{-n} e^{ n \alpha_* R}}{\vol_{n} \left( \mathbb{B}_R^{(n,V)} \right)} \int_{\R^n } \mathbb{1}_{ \left[ \frac{1}{\sqrt{n}} \sum_{i=1}^n \left[ V(x_i) - R\right] \leq 0 \right]} \mathbb{1}_{ \left[\frac{1}{ \sqrt{n}}\sum_{i=1}^n \left[ W(x_i) - m \right] \leq t \right]} e^{\alpha_* \sum_{i=1}^n \left[ V(x_i) - R \right]}dx_1 \ldots d x_n \\
&= 
\frac{Z_{\alpha_*}^{-n} e^{ n \alpha_* R}}{\vol_{n} \left(  \mathbb{B}_R^{(n,V)} \right) } \int_{x=- \infty }^0 \int_{y = - \infty}^t  e^{\sqrt{n}\alpha_* x }d F_n(x,y),
\end{align*}
where $F_n$ is the distribution function of the vector
\begin{equation*}
    \frac{1}{\sqrt{n}} \sum_{i=1}^n \left( V(X_i) - R, W(X_i) - m \right),
\end{equation*}
with $(X_i)_{i \in \N}$ being an iid sequence with $X_1$ being distributed according to $\mu_{\alpha_*, 0}$. By Lemma \ref{Lem_Vol_OB}, it holds that
\begin{equation*}
\vol_{n} \left( \mathbb{B}^{(n,V)}_R \right) = e^{ -\alpha^{*} n R }Z_{\alpha^{*},0}^{n}\frac{1}{\sqrt{2 \pi n} \sigma_{*} | \alpha_* |}(1+ O(n^{-1/2})).
\end{equation*}

Together with Lemma \ref{Lem_Asympt_Int} we now get (let $\Sigma=(\sigma_{i,j})_{i,j=1,2}$ and $\Sigma^{-1} = (\rho_{i,j})_{i,j=1,2}$ and recall the identity $ \rho_{2,2} = \frac{\sigma_{2,2}}{| \Sigma|}$)
\begin{equation}
\begin{split}
\label{Eq_Prob_Clt}
\mathbb{P} \left[ \frac{1}{\sqrt{n}}  \sum_{i=1}^n \left [ W \left( X_i^{(n,V)} \right) - m \right ] \leq t \right] &= \frac{\sqrt{n} | \alpha_* | \sigma_* }{\sqrt{2 \pi } | \Sigma |^{1/2}} \int_{x=0}^{\infty} \int_{y=- \infty}^{t} e^{- \sqrt{n} \alpha_* x } e^{-(x,y) \Sigma^{-1} \left( \begin{smallmatrix}
    x\\y
    \end{smallmatrix} \right) } d x dy + O(n^{-1/2}) \\
&= \frac{\sqrt{\rho_{2,2}}}{\sqrt{2 \pi } } \int_{x=0}^{\infty} \int_{y=- \infty}^{t} e^{-  x } e^{-\left(\frac{x}{\sqrt{n}},y \right) \Sigma^{-1} \left( \begin{smallmatrix}
    \frac{x}{\sqrt{n}}\\y
    \end{smallmatrix} \right) }d y dx + O(n^{-1/2}).
\end{split}
\end{equation}
We take a closer look at the integral above, where we get
\begin{align*}
\int_{x=0}^{\infty} \int_{y=- \infty}^{t} e^{-  x } e^{-\left(\frac{x}{\sqrt{n}},y \right) \Sigma^{-1} \left( \begin{smallmatrix}
    \frac{x}{\sqrt{n}}\\y
    \end{smallmatrix} \right) } d y dx 
    &=    
    \int_{x=0}^{\sqrt{n}} e^{-  x } \int_{y=- \infty}^{t}  e^{-\left(\frac{x}{\sqrt{n}},y \right) \Sigma^{-1} \left( \begin{smallmatrix}
    \frac{x}{\sqrt{n}}\\y
    \end{smallmatrix} \right) }d y dx  + \int_{x=\sqrt{n}}^{\infty} e^{-  x } \int_{y=- \infty}^{t}  e^{-\left(\frac{x}{\sqrt{n}},y \right) \Sigma^{-1} \left( \begin{smallmatrix}
    \frac{x}{\sqrt{n}}\\y
    \end{smallmatrix} \right) }d y dx
\\
&= \int_{x=0}^{\sqrt{n}} e^{-  x } \int_{y=- \infty}^{t}  \exp \left( - \frac{x^2 \rho_{1,1}}{2 n } - \frac{xy \rho_{2,1}}{\sqrt{n} } - \frac{y^2 \rho_{2,2}}{2}\right)d y dx + O \left( n^{-1/2}\right) \\
& \qquad \pm \int_{x=0}^{\sqrt{n}} e^{-x} \int_{y=- \infty}^{t} \exp \left( - \frac{y^2 \rho_{2,2}}{2}\right)dy dx \\
& = 
\int_{x=0}^{\sqrt{n}} e^{-  x } \int_{y=- \infty}^{t}  \exp \left( - \frac{x^2 \rho_{1,1}}{2n} - \frac{xy \rho_{2,1}}{\sqrt{n} } - \frac{y^2 \rho_{2,2}}{2}\right) - \exp \left( - \frac{y^2 \rho_{2,2}}{2}\right)d y dx  \\
& \qquad + \int_{x=0}^{\sqrt{n}} e^{-x} \int_{y=- \infty}^{t} \exp \left( - \frac{y^2 \rho_{2,2}}{2}\right)dy dx + O \left( n^{-1/2}\right). \\
\end{align*}
We aim to show that the first integral in the expression above is of order $O(n^{-1/2})$. We recall the trivial estimate
\begin{equation*}
    \left | e^t - 1 \right| \leq |t|  \left( 1 + e^t \right), \quad t \in \R.
\end{equation*}
This yields (we use that $ \left | \frac{x^2 \rho_{1,1}}{n} + \frac{xy \rho_{2,1}}{\sqrt{n}} \right| \ll \frac{x}{\sqrt{n}} ( |y| +1)$)
\begin{align*}
\left | \int_{x=0}^{\sqrt{n}} e^{-  x } \int_{y=- \infty}^{t}  e^{ - \frac{y^2 \rho_{2,2}}{2}} \left[ \exp \left( - \frac{x^2 \rho_{1,1}}{2n} - \frac{xy \rho_{2,1}}{\sqrt{n} } \right)  - 1 \right ] d y dx \right | 
& \ll \frac{1}{\sqrt{n}}
 \int_{x=0}^{\sqrt{n}} x e^{-  x } \underbrace{\int_{y=- \infty}^{t} (|y| +1) e^{ - \frac{y^2 \rho_{2,2}}{2} - \frac{x y}{\sqrt{n}} } dy }_{ \ll 1}dx \\
 & \ll \frac{1}{\sqrt{n}}  \int_{x=0}^{\sqrt{n}} x e^{-  x } dx \\
 & \ll \frac{1}{\sqrt{n}}. 
\end{align*}
From that we infer
\begin{align*}
 \int_{x=0}^{\infty} \int_{y=- \infty}^{t} e^{-  x } e^{-\left(\frac{x}{\sqrt{n}},y \right) \Sigma^{-1} \left( \begin{smallmatrix}
    \frac{x}{\sqrt{n}}\\y
    \end{smallmatrix} \right) } d y dx &= \int_{x=0}^{\infty} e^{-x} \int_{y=- \infty}^t e^{- \frac{y^2 \rho_{2,2}}{2} } dy dx + O \left( n^{-1/2}\right) \\
    &= \int_{y=- \infty}^t e^{- \frac{y^2 \rho_{2,2}}{2} } dy + O \left( n^{-1/2}\right).
\end{align*}
Combining this estimate with Equation \eqref{Eq_Prob_Clt} gives us
\begin{align*}
\mathbb{P} \left[ \frac{1}{\sqrt{n}}  \sum_{i=1}^n \left [ W \left( X_i^{(n,V)} \right) - m \right ] \leq t \right] &= \phi(t) +  O \left( n^{-1/2}\right),
\end{align*}
where $\phi$ is the cumulative distribution function of a normal distribution with mean $0$ and variance $ \rho_{2,2}^{-1}$. 
\par{}
Next we prove the strong law of large numbers using the convergence Borel--Cantelli lemma. To that end, let $\varepsilon \in (0, \infty)$, where we will show that
\begin{equation*}
    \sum_{n=1}^{\infty} \mathbb{P} \left[ \left | \frac{1}{n} \sum_{i=1}^n W \left( X_i^{(n,V)} \right) - m  \right | > \varepsilon \right] < \infty.
\end{equation*}
Let $n \in \N$ be large, where we have (using an analogous argument as in the first part of this proof)
\begin{align*}
 \mathbb{P} \left[ \left | \frac{1}{n} \sum_{i=1}^n W \left( X_i^{(n,V)} \right) - m  \right | > \varepsilon \right] &=  \int_{x=0}^{\infty} \int_{|y| \geq \sqrt{n} \varepsilon } e^{- \alpha_* \sqrt{n} x }d F_n(x,y) \left( 1 + O \left(n^{-1/2} \right) \right) \\
 & \ll \int_{x=0}^{\infty} \int_{|y| \geq \sqrt{n} \varepsilon } e^{- \alpha_* \sqrt{n} x }d F_n(x,y).
\end{align*}
Since 
\begin{equation*}
  \mathbb{E}_{\mu_{\alpha_*, 0}} \left[ \left | \left | \left(V(X),W(X) \right) \right | \right |_2^3 \right] < \infty,   
\end{equation*}
we can write
\begin{equation*}
    F_n(x,y)= G^{\Sigma}(x,y) + \sum_{j=1}^3 n^{-j/2}F_j(x,y) + R_n(x,y),
\end{equation*}
with $R_n(x,y)=o(n^{-3/2})$, uniformly in $(x,y) \in \R^2$ and the density of $F_j$ is of the form 
$g^{\Sigma}(x,y) \pi_j(x,y)$ with $\pi_j$ having degree $3j$. This implies
\begin{align*}
\int_{x=0}^{\infty} \int_{|y| \geq \sqrt{n} \varepsilon } e^{- \alpha_*  \sqrt{n} x }d F_n(x,y) & = \int_{x=0}^{\infty} \int_{|y| \geq \sqrt{n} \varepsilon } e^{- \alpha_* \sqrt{n} x }d G^{\Sigma}(x,y) +  \sum_{j=1}^3 n^{-j/2} \int_{x=0}^{\infty} \int_{|y| \geq \sqrt{n} \varepsilon } e^{- \alpha_* \sqrt{n} x }d F_j(x,y)  + o(n^{-3/2}) \\
& = \int_{x=0}^{\infty} \int_{|y| \geq \sqrt{n} \varepsilon } e^{- \alpha_* \sqrt{n} x }d G^{\Sigma}(x,y) + o(n^{-3/2}) \\
& \ll \int_{x=0}^{\infty} e^{- \alpha_* \sqrt{n} x } \int_{|y| \geq \sqrt{n} \varepsilon } e^{- c ( x^2 + y^2) }  d y dx + o(n^{-3/2})\\
& \ll 
\int_{x=0}^{\infty} e^{- \alpha_* \sqrt{n} x - cx^2 } e^{- c' n} dx + o(n^{-3/2}) \\
& = o(n^{-3/2}),
\end{align*}
where, in the second last line, we used that
\begin{equation*}
\int_{|y| \geq \sqrt{n} \varepsilon } e^{- c  y^2 }  d y  \ll e^{-c' n} ,
\end{equation*}
for some constant $c' \in (0, \infty)$. This implies that
\begin{equation*}
\sum_{n=1}^{\infty} \mathbb{P} \left[ \left | \frac{1}{n} \sum_{i=1}^n W \left( X_i^{(n,V)} \right) - m  \right | > \varepsilon \right] = \sum_{n=1}^{\infty} o(n^{-3/2}) < \infty.    
\end{equation*}
\end{proof}

\begin{proof}(of Corollary \ref{Cor_Clt})
We recall that
\begin{align*}
\frac{\vol_{n}\left( \mathbb{B}_{R}^{(n,V)} \cap \mathbb{B}_{m}^{(n,W)}\right)}{\vol_{n} \left(\mathbb{B}_{R}^{(n,V)} \right)} &= \mathbb{P} \left[  \frac{1}{n} \sum_{i=1}^n W \left( X_i^{(n,V)} \right) \leq m \right] \\
&= \mathbb{P} \left[  \frac{1}{\sqrt{n}} \sum_{i=1}^n \left[ W \left( X_i^{(n,V)} \right) - m \right] \leq 0 \right],
\end{align*}
where $X^{(n,V)}$ is uniformly distributed on $\mathbb{B}_R^{(n,V)}$. The last probability in the expression above converges to $\frac{1}{2}$ by Theorem \ref{Thm_SLLN_CLT}, which proves the claim.
\end{proof}

% \begin{proof}(of Corollary \ref{Cor_subquad})
% If $V$ satisfies $\lim_{x \rightarrow \infty} \frac{V(x)}{x^2}= \infty$, the claim follows from \cite[Section 3.3.1]{kim2022}. 
% Thus we only need to treat the case where $\lim_{x \rightarrow \infty} \frac{V(x)}{x^2}=0$, or equivalently $\lim_{x \rightarrow \infty} \frac{x^2}{V(x)}= \infty$. Let $[a,b ] \subseteq \R_+$ be an arbitrary interval, then for $X^{(n,V)} \sim \U \left( \mathbb{B}_R^{(n,V)} \right)$ we have
% \begin{align*}
%     \mathbb{P} \left[ \left | \left | X^{(n,V)} \right | \right |_2 \in [a,b]  \right] & =  \mathbb{P} \left[ \left | \left | X^{(n,V)} \right | \right |_2 \leq b  \right] - \mathbb{P} \left[ \left | \left | X^{(n,V)} \right | \right |_2 \leq a  \right] \\
%     & =: b_n - a_n.
% \end{align*}
% For $b_n$ we get
% \begin{align*}
%     b_n &= \mathbb{P} \left[ \left | \left | X^{(n,V)} \right | \right |_2 \leq b  \right] \\
%     &= \frac{\vol^{(n)}\left( \mathbb{B}^{(n,2)}_b \right)}{\vol^{(n)}\left( \mathbb{B}^{(n,V)}_R \right)} \mathbb{P} \left[ \frac{1}{n} \sum_{i=1}^n V\left( X_i^{(n,2)} \right) \leq R \right],
% \end{align*}
% where $X^{(n,2)} \sim \U \left( \mathbb{B}^{(n,2)}_b \right)$ and $ \mathbb{B}^{(n,2)}_b$ denotes the $\ell^2$ ball with radius $b$.
% \end{proof}

% % % % % % % % % % % % % % % %
\subsection*{Acknowledgement}
% % % % % % % % % % % % % % % %
LF and JP are supported by the Austrian Science Fund (FWF) Project P32405 \textit{Asymptotic geometric analysis and applications}. LF is also supported by the FWF Project P 35322 \textit{Zufall und Determinismus in Analysis und Zahlentheorie}.

\bibliographystyle{plain}
\bibliography{bibliographie.bib}

\bigskip
\bigskip

\medskip

\small

\noindent \textsc{Lorenz Fr\"uhwirth:} Faculty of Computer Science and Mathematics, University of Passau, Dr.-Hans-Kapfinger-Straße 30, 94032 Passau, Germany. 

\noindent
\textit{E-mail:} \texttt{lorenz.fruehwirth@math.tugraz.at}

\medskip

\small

\noindent \textsc{Joscha Prochno:} Faculty of Computer Science and Mathematics, University of Passau, Dr.-Hans-Kapfinger-Straße 30, 94032 Passau, Germany. 

\noindent
\textit{E-mail:} \texttt{joscha.prochno@uni-passau.de}

\end{document}